\documentclass[a4paper, 12pt,reqno]{amsart}
\usepackage[T1]{fontenc}
\usepackage{amsmath,amssymb,amsthm,eucal,manfnt,amsfonts, amsmath,amscd, amssymb, latexsym, mathrsfs, verbatim, wasysym,marvosym}%stmaryrd, ifsym
\usepackage{color, tikz}
\usepackage[all]{xy}
\usepackage{cite}
\usepackage{sidecap}
\usepackage{tikz}
\usepackage{graphicx}
\usepackage{lmodern}
\usepackage[colorlinks]{hyperref}
\usepackage{verbatim}
\usepackage[colorinlistoftodos]{todonotes}
\usepackage{braket,bm,enumitem}
\usepackage{mathtools}
\usepackage{stackrel}
\usepackage{booktabs}
\usepackage{multirow}
 \usepackage{url}
\DeclareRobustCommand{\rchi}{{\mathpalette\irchi\relax}}
\newcommand{\irchi}[2]{\raisebox{\depth}{$#1\chi$}} % inner command, used by \rchi

\renewcommand{\Set}[1]{\left\{#1\right\}}

\definecolor{antiquewhite}{rgb}{0.98, 0.92, 0.84}

\DeclareFontFamily{OT1}{pzc}{}
\DeclareFontShape{OT1}{pzc}{m}{it}{<-> s * [1.10] pzcmi7t}{}
\DeclareMathAlphabet{\mathpzc}{OT1}{pzc}{m}{it}

\def\np{\par\noindent}
\def\bc{{\mathbb C}}
\def\bn{{\mathbb N}}

\def\br{{\mathbb R}}
\def\bz{{\mathbb Z}}
\def\zcW{\mathpzc{W}} 
\def\cb{{\mathcal B}}

\def\ch{{\mathcal H}}

\def\cs{{\mathcal S}}
\def\ct{{\mathcal T}}
\def\a{\alpha}

\def\G{\Gamma}
\def\g{\gamma}

\def\r{\rho}
\def\s{\sigma}
\def\t{\tau}
\def\vecLength{\boldsymbol{\ell}}

\title[A Note on Twisted Crossed Products and Spectral Triples]{A Note on Twisted Crossed Products \\ and Spectral Triples}
\date{\today}

\author{P. Antonini}
\address{Universit\`a del Salento, I-73100 Lecce}
\email{paolo.antonini@unisalento.it}

\author{D. Guido}
\address{Universit\`a degli Studi di Roma \textquotedblleft Tor Vergata\textquotedblright\,  I-00133 Roma }
\email{guido@mat.uniroma2.it}

\author{T. Isola}
\address{Universit\`a degli Studi di Roma \textquotedblleft Tor Vergata\textquotedblright\,  I-00133 Roma }
\email{isola@mat.uniroma2.it}

\author{A. Rubin}
\address{Scuola Internazionale Superiore di Studi Avanzati (SISSA),  I-34136 Trieste}
\email{alrubin@sissa.it}

\thanks{}
\subjclass[]{}

\advance\topmargin by -\headheight
\advance\topmargin by -\headsep
\numberwithin{equation}{section} %% needs `amsmath' package

\theoremstyle{plain} %% needs `amsmath' package
\newtheorem{thm}{Theorem}[section]
\newtheorem{setup}{Setup}[section]
\newtheorem{lemma}[thm]{Lemma}
\newtheorem{prop}[thm]{Proposition}

\theoremstyle{definition} %% needs `amsmath' package
\newtheorem{defn}[thm]{Definition}
\newtheorem{exa}[thm]{Example}

\newtheorem*{notat}{Conventions}
\newtheorem*{ack}{Acknowledgments}

\theoremstyle{remark} %% needs `amsmath' package
\newtheorem{rmk}[thm]{Remark}

\DeclareMathOperator{\Dom}{Dom}   %% domain of an operator
\DeclareMathOperator{\End}{End}   %% endomorphism algebra
\DeclareMathOperator{\tr}{tr}     %% matrix trace
\newcommand{\eps}{\varepsilon} %% short for \varepsilon
\newcommand{\sg}{\sigma}      %% short for \sigma
\newcommand{\f}{\varphi}      %% short for \varphi
\newcommand{\rest}[1]{\big\rvert_{#1}} % restriction e.g. to boundary

\newcommand{\A}{\mathcal{A}}  %% an algebra
\newcommand{\C}{\mathbb{C}}   %% complex numbers
\newcommand{\D}{\mathcal{D}}  %% a selfadjoint operator
\renewcommand{\H}{\mathcal{H}}  %% a Hilbert space
\newcommand{\N}{\mathbb{N}}   %% natural numbers
\newcommand{\Tt}{\mathcal{T}} %%Toeplitz
\newcommand{\Z}{\mathbb{Z}}   %% integers
\newcommand{\bT}{\mathbb{T}}     %% 1-dim torus
\def\Ind{\operatorname{Ind}}  %% index of an element

\def\End{\operatorname{End}}

\newcommand{\norm}[1]{\left\lVert#1\right\rVert}
\newcommand{\abs}[1]{\left\lvert#1\right\rvert}

\newcommand\xqed[1]{%
	\leavevmode\unskip\penalty9999 \hbox{}\nobreak\hfill
	\quad\hbox{#1}}
\newcommand\demo{\xqed{$\Box$}}

\usepackage{color}
\definecolor{darkgreen}{cmyk}{1,0,1,.2}
\definecolor{m}{rgb}{1,0.1,1}
\definecolor{green}{cmyk}{1,0,1,0}
\definecolor{darkred}{rgb}{0.55, 0.0, 0.0}
\definecolor{test}{rgb}{1,0,0}
\definecolor{cmyk}{cmyk}{0,1,1,0}

\begin{document}

	\begin{abstract}
		Starting with a spectral triple on a unital $C^{*}$-algebra $A$ with an action of a
		discrete group $G$, if the action is uniformly bounded (in a Lipschitz sense) a spectral triple on the reduced crossed product $C^{*}$-algebra $A\rtimes_{r} G$ is constructed in \cite{hawkins2013spectral}. The main instrument is the Kasparov external product. We note that this construction still works for twisted crossed products when the twisted action is uniformly bounded in the appropriate sense. Under suitable assumptions we discuss some basic properties of the resulting triples: summability and regularity. Noncommutative coverings with finite abelian structure group
				 are among the most basic, still interesting, examples of twisted crossed products; we describe their main features.
	\end{abstract}

	\maketitle

	\tableofcontents

	%%%%%%%%%%%%%%%%%%%%%%%%%%%%%%%%%%%%%%%%%%%%%%%%%%%%%%%%%%%%%%%%%%%%%%%%%%%%%%%%%
	\section{Introduction}
\noindent 	
Spectral triples on $C^*$-algebras are a central notion in noncommutative geometry, being modeled on the geometric structure  codified by the properties of a Dirac type operator on a smooth manifold. The range of applicability 
	of this paradigm is vast, going from the foundational example of Spin manifolds to foliated manifolds, group $C^*$-algebras, quantum groups, quantum deformations and  fractals. This list is not exhaustive and we address the reader to \cite{connes1994noncommutative,gracia2013elements,connes1995noncommutative,connes1996gravity,connes2013spectral,connes2019noncommutative}, the recent works \cite{MR4328052,MR4358538} and the references therein.	
	
	Following the stream pursued by many authors \cite{hawkins2013spectral,paterson2014contractive,bellissard2010dynamical,iochum2014crossed}, in this paper we construct and study spectral triples on reduced twisted crossed products $A\rtimes^{\rho,\sigma}_rG$, where $A$ is a  unital $C^*$-algebra, $G$ a discrete group and $(\rho,\sigma)$ a twisted action in the sense of Busby and Smith \cite{busby1970representations}. Two facts brought us to consider twisted crossed products: on the one hand, the crossed product construction is fundamental in  the theory of $C^*$-algebras and in
		noncommutative topology since it replaces the operation of forming a quotient when this is a singular, bad-behaved space.
		On the other hand, the study of noncommutative coverings  with finite abelian structure group (that two of the authors developed in \cite{aiello2017spectral} as \emph{building blocks} of solenoidal noncommutative metric spaces) brings naturally to the conclusion that in the case of finite abelian groups, any twisted crossed product is a regular covering and any regular covering is a twisted crossed product (Theorem \ref{isocrossed}, previously proven with a different language in \cite{wagner2015noncommutative}). 	
		This makes the problem of the construction of spectral triples on crossed products a natural subject of interest both for the creation of a new large class of examples and for the analysis of the Connes axioms in the case of noncommutative quotient manifolds \cite{connes1996gravity,connes1995noncommutative,connes2013spectral}.

	In constructing a spectral triple on a twisted crossed product, we follow, as in \cite{hawkins2013spectral}, the guiding principle of the Kasparov external product, combining the given Dirac operator on $A$ with a length function on the group. We have decided to utilize matrix valued length functions  because, as pointed out in \cite{hawkins2013spectral}, these turn out to be useful in the procedure of iterating crossed products. \\
	
	We pass now to describe the content and the results of the paper  giving more details and some comments. 
	
	In section 2 we follow \cite{busby1970representations,packer1989twisted, bedos2011discrete, thiang2016k} and give a concise account on the basic theory of twisted actions of discrete groups on unital $C^*$-algebras, twisted covariant representations and  the twisted crossed product construction.
	 The choice of considering the unital case is dictated by the interest in constructing unital spectral triples and simplifies a bit the presentation. 
	 
 In section 3 we construct spectral triples on twisted crossed products. We begin by discussing generalities about spectral triples and translation bounded and proper matrix valued length functions. Here and in the following $V$ will always denote a finite dimensional Hermitian vector space so that matrix valued length functions will be maps
	  $\bm{\ell}:G \rightarrow \mathbb{B}_h(V)$ on the discrete group $G$ and valued in the finite dimensional selfadjoint operators.
	  Such maps have been
	   introduced in \cite{hawkins2013spectral} and generalise scalar length functions on groups. With them we can create spectral triples on the group algebra $\mathbb{C}G$ of the form $(\mathbb{C}G,\ell^2(G) \otimes V,M_{\bm{\ell}})$ with Dirac operator having as coefficients 
	   elements of $\mathbb{B}_h(V)$.
	 The triple $(\mathbb{C}G,\ell^2(G) \otimes V,M_{\bm{\ell}})$ is the prototypical triple on a crossed product and we study  its summability properties. 
	 To do so, we first propose 
	 (definition \ref{Pgrowth})
	 the notion of polynomial growth for a discrete (finitely generated) group with respect to a length function which is matrix valued. This naturally extends the usual concept of polynomial growth with respect to some word metric.
	 We then show (proposition \ref{growth}) that 
	 the triple 
$(\mathbb{C}G,\ell^2(G) \otimes V,M_{\bm{\ell}})$ is
	finitely summable if and only if the group has polynomial growth with respect to $\bm{\ell}$.
	
	The construction of spectral triples on twisted crossed products is the content of Theorem \ref{spectraltripleconstruction} and generalizes the one in \cite{hawkins2013spectral}. 
	Given a group with a proper, translation bounded matrix valued length function $(G,\bm{\ell})$, a twisted action $(\rho,\sigma)$ of $G$ on the $C^*$-algebra $A$ and a spectral triple $(\mathcal{A},H,D)$  on $A$, our basic assumptions (summarized in the setup \ref{setuptwisted}) to  extend this triple to $A\rtimes_{r}^{(\rho,\sigma)}G$ are the following:
	\begin{enumerate}
	\item the twisted action is smooth. This means that the twisting pair $(\rho,\sigma)$ is valued in the smooth algebra,
	\item a uniform boundedness condition for the action:
	\begin{equation}\nonumber
			\sup_{x\in G}\| [D,\pi(\rho_x(a)\sigma_{x,y})]\|<\infty
		\end{equation}
 is satisfied	for every $a\in \mathcal{A}$ and $y\in G$. 
			\end{enumerate}
The second condition above is the twisted version of the one appearing in \cite{hawkins2013spectral}; it
ensures boundedness of the commutators between the Dirac operator and the elements of the (smooth) twisted convolution algebra. Moreover, we prove that if the starting triple on $A$ is finitely summable and the group has polynomial growth with respect to the length function, then also the extended triple is finitely summable. The abscissa of convergence is bounded by the sum of the abscissa of convergence of the starting triple and the group growth.

In the remaining part of the section we examine the particularly favorable, still interesting, situation in which  the representation $\pi:A \rightarrow \mathbb{B}(H)$ associated to the starting triple $(A,H,D)$ is also twisted covariant with respect to the twisted action $\rho$ of $G$ on $A$ and the Dirac operator has bounded commutators with the unitaries $U_g$ implementing the group action on $H$. 
Under this assumption, we describe what we call the \emph{equivariant construction}. This was considered in remark 2.9 in {\cite{hawkins2013spectral}},
where it was related to the equivariant Kasparov product. 
In this case a second construction of a spectral triple on a twisted crossed product can be given. It differs from the previous one only in the representation so that, under a condition of uniform boundedness of the commutators between the Dirac operator and the unitaries $U_g$, this second construction is a bounded perturbation of the first one (up to unitary equivalence). In particular the $K$-homological information carried by the two is unchanged.
Using the natural morphism $\phi_A^r:A\rtimes_r^{\rho,\sigma}G \rightarrow (A\rtimes^{\rho,\sigma}G)\otimes C^*_r(G)$,
we notice 
(proposition \ref{productrestriction})
that the spectral triple obtained in the equivariant case can be seen as the restriction of an external product of two spectral triples: the first one is a triple on the algebra $B$ given by the image of the twisted crossed product under the integrated form of the starting representation, and the second one is the (standard) spectral triple $(C^*_r(G),\ell^2(G) \otimes V,M_{\bm{\ell}})$.

In section 4 we continue the study of the spectral triple constructed in the equivariant case. Since the triple is a restriction of an exterior product and the product of regular triples is regular, we can examine its regularity properties. After giving some generalities about regularity of spectral triples, we prove theorem \ref{maintheorem} giving sufficient conditions ensuing the regularity of the spectral triple in the equivariant construction.
	The regularity of a spectral triple being a strong property, these assumptions are, though natural, quite restrictive and demand a strict compatibility 
	between $D$ and the group action. Nevertheless, interesting examples can be found even in the case of ordinary (untwisted) crossed products.

Section 5 is devoted to our main example of a twisted crossed product: noncommutative coverings with finite abelian structure group,  	which consist of an action of a finite abelian group $G$ on a $C^*$-algebra $B$. Then we say that $B$ \emph{covers} the algebra of fixed points $A:=B^{G}$.  We study in detail the structure of these noncommutative coverings, reserving particular attention to the notion of {\em{rank}-$1$ \em{regularity}}, a property formulated in terms of the spectral subspaces of the action, that was previously called {\em{regularity}} in \cite{aiello2017spectral}. In particular, we relate it with the concept of freeness of the group action and we give examples. When the action is \em{rank}-$1$ \em{regular}, a twisted action of the dual group $\widehat{G}$ can be found so that $B\cong A\rtimes_r^{\rho,\sigma} \widehat{G}$. The last section gives an example of the construction of a spectral triple on a twisted crossed product.

	%	that if $U_g \in \operatorname{Op}^0(\Delta)$, if   $[D,U_g]\in \operatorname{Op}^0(\Delta)$, 	
%$$ 
%[\underbrace{\Delta,[\Delta, [\Delta, \dots [\Delta}_{k\textrm{\, times}},U_g  ]\,]\,]\in \operatorname{Op}^k(\Delta) \quad 	\textrm{and}\quad  [\underbrace{\Delta,[\Delta, [\Delta\dots}_{k\textrm{\, times}}\,,  [D,U_g  ]\,]\,]\in \operatorname{Op}^k(\Delta)
%$$ 

	Appendix A contains some facts  about the infinitesimal order of compact operators that have been used in section \ref{sec:4}, and Appendix B deals with the notion of regularity for spectral triples.
		\begin{notat}
		All Hilbert spaces and $C^{*}$-algebras in this paper are separable; all algebras are complex and unital. 
	\end{notat}

	%%%%%%%%%%%%%%%%%%%%%
\section{Twisted Crossed Products}\label{Appendix 1}
We briefly recall some basic facts about twisted crossed products following mainly \cite{busby1970representations,packer1989twisted, bedos2011discrete, thiang2016k}. We denote by $A$ a separable unital $C^{*}$-algebra and by $G$ a discrete group with neutral element $e$. 

\begin{defn}\label{covrep}A \emph{twisting pair} $\left(\rho,  \sigma\right)$ for $A$ and $G$ is a couple of maps $\rho\colon G\rightarrow \textup{Aut}(A)$ and $\sigma\colon G\times G\rightarrow U(A)$ satisfying:
\begin{enumerate}  
	\item $ \rho_x(\sigma_{y,z}) = \sigma_{x,y}\sigma_{xy,z}\sigma_{x,yz}^*$
	\item $ \rho_x \circ \rho_y = \operatorname{Ad}(\sigma_{x,y}) \circ \rho_{xy}$  
	\item $\sigma_{x,e}=\sigma_{e,x}=1$, $\rho_e=\operatorname{id}_A$.
\end{enumerate}
for all $x,y,z \in G$.	We call $(A,G,\rho,\sigma)$  a \emph{twisted dynamical system}.	The map $\sigma$ is called a \emph{2-cocycle} with  values in $U(A)$.
\end{defn}

Given a twisted dynamical system $(A,G,\rho,\sigma)$, we define a twisted convolution $\ast$-algebra structure on the space
$C_c(G,A)$
of the finitely supported functions from $G$ to $A$ by
$$(f \star g)(x)=\sum_{y\in G}f(y)\rho_y(g(y^{-1}x))\sigma_{y,y^{-1}x}, \quad f^*(x)=\sigma_{x,x^{-1}}^* \rho_{x}(f(x^{-1})^*),$$
for $f,g \in C_c(G,A)$.
In terms of the $\delta$-functions, for $x,y \in G$ we have:
$$(a\delta_x) \star (b\delta_y)=a\rho_x(b)\sigma_{x,y} \delta_{xy}, \quad (a\delta_x)^*=\sigma^*_{x^{-1},x}\rho_{x^{-1}}(a^*) \delta_{x^{-1}}.$$ 
Such rules extend to $L^1(G,A)$ making it the Banach $\ast$-algebra $L^1(A,G,\rho,\sigma)$. 

\begin{defn}
	A \emph{twisted covariant representation} of $(A,G,\rho,\sigma)$ on a Hilbert space $H$ is a couple $(\pi,U)$, where $\pi\colon A\rightarrow \mathbb{B}(H)$ is a non degenerate representation  and $U:G \rightarrow U(H)$ satisfies 
	\begin{enumerate}
		\item $U_x  U_y =\pi(\sigma_{x,y})  U_{xy}$,
		\item $
	\pi \circ \rho_x= \operatorname{Ad}(U_x) \circ \pi,$ 
	\end{enumerate}
	for every $x,y \in G$.
\end{defn}
It follows from the first relation that $U_e=\operatorname{Id}$. We have a $1-1$ correspondence between twisted covariant representations and non degenerate representations of $L^1(A,G,\rho,\sigma)$: to the twisted covariant representation $(\pi,U)$ one associates its {\em{integrated form}}
$\Pi: L^1(A,G,\rho,\sigma) \rightarrow \mathbb{B}(H)$ such that
\begin{equation}\label{intform}
	\Pi \Big{(}\sum_{x\in G} a_x\delta_x \Big{)}:= \sum_{x\in G}\pi(a_x) U_x.
\end{equation}

\begin{defn}
	A twisted crossed product for $(A,G,\rho,\sigma)$ is the following universal object: a $C^*$-algebra $B$ equipped with a unital morphism $i_{A}:A \rightarrow B$ and a map $i_{G}:G \rightarrow U(B)$ such that
	\begin{enumerate}
		\item The couple $(i_{A},i_{G})$ is covariant in the sense that
		\begin{displaymath}
		\quad i_{G}(x)i_{G}(y)=i_{A}(\sigma_{x,y})i_{G}(xy)
\quad \textrm{and} 	\quad 
			i_{A} \circ \rho_x = \operatorname{Ad}(i_{G}(x))\circ i_{A} \quad 	\end{displaymath} 
		for every $x,y \in G.$
		\item For every covariant representation $(\pi,U)$ of $(A,G,\rho,\sigma)$ on the Hilbert space $H$ there is a non degenerate representation $\pi \times U :B \rightarrow \mathbb{B}(H)$ making the diagram
		\begin{displaymath}
			\xymatrix{G\ar@/_1pc/[dr]_{U}\ar[r]^{i_{G}} & B\ar[d]^{\pi \times U} & A\ar@/^1pc/[dl]^{\pi}\ar[l]_{i_{A}} \\
				{}&\mathbb{B}(H)& {}}
		\end{displaymath}
		commute. 
			\end{enumerate}
\end{defn}

A twisted crossed product exists and is unique  up to isomorphisms which take into account also the couple $(i_{A},i_{G})$ (see \cite{packer1989twisted}). It will be denoted with $A \rtimes^{\rho,\sigma} G$ or, when the context is clear, just with $A \rtimes
G$. Moreover, under our assumptions (making the algebra $L^1(A,G, \rho, \sigma)$ unital), it is generated by the image of the integration map 
\begin{displaymath}
	i_{A} \times i_{G}: L^1(A,G, \rho, \sigma) \rightarrow A \rtimes^{\rho,\sigma}G, \quad i_{A}\times i_{G}\left(\sum_{x\in G} a_x\delta_x \right)=\sum_{x\in G}i_{A}(a_x)i_{G}(x)
\end{displaymath}
and coincides with the enveloping $C^*$-algebra $C^*(L^1(A,G, \rho, \sigma))$. In particular the representation $\pi \times U$ induced by 
any covariant representation is equal, on the image of $i_{A}\times i_{G}$, to the representation \eqref{intform} and will be called the \emph{integrated form} of $(\pi,U)$ as well.

\begin{exa}[Clifford Algebras]As proved in \cite{albuquerque2002clifford}, we can regard (complex) Clifford algebras as twisted crossed products of $\mathbb{C}$ with a suitable finite group. Indeed, $\mathbb{C}\ell_{n}$ is the complex algebra generated by the skew-adjoint anti-commuting elements $e_{1}, \dots, e_{n}$ such that $e^{2}_{i}=-1$ and graded in the standard way. Consider the cyclic multiplicative group $\mathbb{Z}_{2}$ and let $g=-1$ be its generator; any element in $G=\mathbb{Z}_{2}^{n}$ is of the form $x = (g^{x_{1}},\dots ,g^{x_{n}})$ where $x_{i}=0,1\in \mathbb{Z}$ for $i=1,\dots,n$. Define 
	\begin{displaymath}
		\sigma_{n}(x,y) = (-1)^{\sum_{j< i}x_{i}y_{j}}
	\end{displaymath}
	for $x,y\in G$ and let $\mathbb{C}\rtimes^{\textup{id},\sigma_{n}}G$ be the twisted group algebra for the trivial action of $G$ on $\mathbb{C}$. Then the map $\mathbb{C}\rtimes^{\textup{id},\sigma_{n}}G\rightarrow \mathbb{C}\ell_{n}$, given on a basis by
	\begin{displaymath}
		(g^{x_{1}},\dots, g^{x_{n}})\longmapsto (ie_{1})^{x_{1}}\cdots (ie_{n})^{x_{n}},
	\end{displaymath}
	is an isomorphism of $C^{*}$-algebras. 
		\demo
\end{exa}

 We will be concerned with the {\em{reduced twisted crossed product}} which is defined by a particular class of representations.
One starts with a faithful representation $\pi:A \rightarrow \mathbb{B}(H)$, then a canonical 
covariant representation $(\widetilde{\pi},\widetilde{R})$ on the Hilbert space $\ell^2(G,H)$ is defined by 
\begin{displaymath}
	\widetilde{\pi}(a)(\xi \otimes \delta_x)=\pi(\rho_x(a))\xi \otimes \delta_x, \quad \widetilde{R}_x(\xi\otimes \delta_y)=\pi(\sigma_{yx^{-1},x})\xi\otimes \delta_{yx^{-1}} 
\end{displaymath} given on simple tensors of $\ell^2(G,H)\simeq H\otimes \ell^{2}(G)$ and for $a\in A$.
This is the {\em{right induced covariant representation}};
its integrated form $\widetilde{\pi}\times \widetilde{R}: A\rtimes^{\rho,\sigma} G \rightarrow \mathbb{B}(\ell^2(G,H))$ reads as
\begin{displaymath}
	\widetilde{\pi} \times \widetilde{R}\Big{(}\sum_{x\in G}a_x\delta_x\Big{)} (\xi \otimes \delta_y)= \sum_{x\in G}\pi(\rho_{yx^{-1}}(a_x)\sigma_{yx^{-1},x})\xi \otimes \delta_{yx^{-1}},
\end{displaymath}
on the elements $\sum_{x\in G}a_x\delta_x \in L^1(A,G,\rho,\sigma)$. It is called the {\em{regular representation}} of the twisted crossed product.
\begin{defn}
	The reduced twisted crossed product
	$A \rtimes^{\rho,\sigma}_r G$ (or just $A \rtimes_r G$)
	is the $C^*$-algebra image of  the regular representation.
\end{defn}
It can be shown that the reduced crossed product is independent of the choice of the faithful representation $\pi$ \cite[remark 3.12]{packer1989twisted}.\\

We have presented the theory with the right induced representations as in  \cite{busby1970representations,packer1989twisted}. However, for the construction of spectral triples it will be useful, to reconcile with the existing literature, to consider a different version of the regular representation. This will be obtained by conjugation with the unitary involution $V_{G}\colon\ell^2(G,H) \rightarrow \ell^2(G,H)$ given by $V_{G}(\xi\otimes \delta_x)=\xi\otimes \delta_{x^{-1}}$. We get in this way, 
out of the faithful representation $\pi:A\rightarrow \mathbb{B}(H)$,
the representation
$\Pi: A \rtimes^{\rho,\sigma} G \rightarrow \ell^2(G,H),$ given by
\begin{equation}\label{left}
	\Pi(a\delta_x)(\xi\otimes \delta_y)= \pi(\rho_{y^{-1}x^{-1}}(a)\sigma_{y^{-1}x^{-1},x})\,\xi \otimes \delta_{xy},
\end{equation}
which is the integrated form of  
a covariant representation  $(\widetilde{\pi},\widetilde{L})$ with
\begin{equation}\label{leftpair}
	\widetilde{\pi}(a)(\xi \otimes\delta_x)=\pi(\rho_{x^{-1}}(a))\xi\otimes \delta_x, \quad \widetilde{L}_x (\xi\otimes \delta_y)=\pi(\sigma_{y^{-1}x^{-1},x})\xi \otimes \delta_{xy}.
\end{equation}

So in the following
we shall consider the reduced crossed product as defined by $\Pi$ and call it the {\em{left regular representation}} (induced by $\pi$). 
This version of the reduced crossed product is of course isomorphic to the one given before.

Observe that the previous construction of ($\widetilde{\pi}$,$\widetilde{L}$) can be done even though the representation $\pi$ is not faithful. However, in this case, it does not define the reduced crossed product.

%%%%%%%%%%%%%%%%%%%%%%%%%%%%%%%%%%%%%%%%%%%%%%%%%%%%%%%%%%%%%%%%%%%%%%%%%%%%%%%%%%%%%%%%%%%%%%%%%%%%%%%%%%%%%%%%%%%%%%%%%%%%%%%%%%%%%%%%	
\section{Spectral Triples on Twisted Crossed Products}\label{sec:4}
\noindent In this section we show how to construct a spectral triple on a twisted crossed product $A \rtimes^{\rho,\sigma}_r G$ starting from a spectral triple on $A$ satisfying some assumptions.
The idea, as in \cite{hawkins2013spectral}, is to borrow the Dirac operator from an unbounded representative for the external Kasparov product with a spectral triple on the group algebra $\mathbb{C}{G}$ which is constructed with a (matrix valued) length function.

\begin{defn}
\begin{enumerate}
\item
	An \textit{odd spectral triple} $(\mathcal{A}, H, D)$ on a unital $C^{*}$-algebra $A$ consists of a dense $*$-subalgebra $\mathcal{A}\subseteq A$ represented by $\pi\colon A\rightarrow \mathbb{B}(H)$ on a Hilbert space $H$ and a self-adjoint operator $D$ (called a \emph{Dirac operator})
	densely defined on $\Dom{D}\subset H$  such that 
	$(1+D^2)^{-\frac{1}{2}}$ is compact, 
	$\pi(a)(\Dom{D})\subseteq \Dom{D}$ and the commutator
	$[D,\pi(a)]$ extends to a bounded operator on $H$ for every  
	$a\in \mathcal{A}$.
	
\item
An \emph{even} spectral triple on $A$ is given by the same data with the addition of a $\mathbb{Z}_{2}$-grading, namely a self-adjoint operator $\rchi \colon H\rightarrow H$ called a \emph{grading operator} such that
$\chi^{2}=1$,
$\pi(a)\rchi = \rchi \pi(a)$ for all $a\in A$, 
$\rchi(\Dom{D})\subseteq \Dom{D}$ and $D\rchi = -\rchi D$. In this case it is possible to split $H= H^+\oplus H^-$  and
\begin{displaymath}
	\rchi=\left(\begin{matrix}
		1 & 0 \\
		0 & -1
	\end{matrix}\right),\qquad 
	\pi =  \left(\begin{matrix}
		\pi^+ & 0 \\
		0 & \pi^-
	\end{matrix}\right), \qquad D= \left(\begin{matrix}
		0 & D^- \\
		D^+ & 0
	\end{matrix}\right).
\end{displaymath}
\item
	A spectral triple $(\mathcal{A}, H, D)$ on a unital $C^{*}$-algebra $A$ is called  \emph{non-degenerate} when  the representation $\pi$ is faithful and $[D, \pi(a)]=0$ 	if and only if $a\in \mathbb{C}1_{A}$. 
\item
	Given a spectral triple $(\A,\H,D)$, we consider the function 
	\begin{displaymath}
		\zeta_D(t)\coloneqq\tr ((I+D^2)^{-t/2})
	\end{displaymath}for $t\in \mathbb{C}$. The (complex) numbers $t$ for which $\zeta_D(t)$ is finite  are called \emph{summability exponents} and the triple is called \emph{finitely summable} if there exists such a number. In this case, we set 
	\begin{displaymath}
		\textup{abs} (\zeta_D)\coloneqq \inf \{t>0:\zeta_D(t)<\infty\}
	\end{displaymath}and call it \emph{abscissa of convergence} of $\zeta_D$.
\end{enumerate}
\end{defn}
If the representation $\pi$ is not clear from the context, we will use the notation  $(A, H, D, \pi)$.  
We recall that $\textup{abs} (\zeta_D)$ is the unique number $d$, if any, for which the Dixmier logarithmic trace $\tr_\omega ((I+D^2)^{-d/2})$ is finite non zero (cf. e.g. \cite{GuIs9}  theorem 2.7).

\subsection{Length Functions and Spectral Triples on Groups}

Let $V$ be a finite dimensional Hilbert space
(usually taken even dimensional) 
and denote by $\mathbb{B}_h(V)$  the space of selfadjoint operators on $V$. 
\begin{defn}\label{generalisedlenght}
	A function $\bm{\ell}: G \to \mathbb{B}_h(V)$ is called
	\begin{itemize}
		\item \emph{translation bounded} if for every $y\in G$ the translation function 
		\begin{equation}\label{tbf}
			\bm{\ell}_y(x)=\bm{\ell}(x)-\bm{\ell}(y^{-1}x)
		\end{equation} is bounded (in the $x$ variable),
		\item \emph{proper} if 
		\begin{enumerate}
			\item 	${\bm{\ell}}(g)=0$ if and only if $g=0$. \label{propercondition}
			\item the union of all the spectra $\mathcal{S}=\bigcup_{x \in G}\operatorname{Sp} \bm{\ell}(x)$ is a discrete set in $\mathbb{R}$ 
			\item any $t \in \mathcal{S}$ corresponds only to a finite number of elements $x \in G$.
		\end{enumerate}
	\end{itemize}
\end{defn}
Here and in the following with $\operatorname{Sp}$ we denote the spectrum of an operator.
Notice that the previous definition contains the standard case of proper translation bounded scalar length functions $\bm{\ell} : G \to \mathbb{Z}$.

\begin{exa}
	Let us construct such functions on the group $G=\mathbb{Z}^{n}$. Let $\iota:\mathbb{Z}^n \hookrightarrow \mathbb{R}^n $ be a norm preserving group embedding. Then for any 
	unitary\footnote{Clifford algebras conventions: $\mathbb{C}\bm{\ell}_n$ is the complex Clifford algebra where $vw+wv= - \langle v,w \rangle$ for $v,w \in \mathbb{R}^n$. 
		A Clifford representation is unitary (a.k.a. \emph{self adjoint}) if $\varepsilon({v})$ is unitary for $\|v\|=1$. Then it follows that $\varepsilon(v)^*=-\varepsilon(v)$  } Clifford representation 
	$\varepsilon: \mathbb{C}{\ell}_n \to \operatorname{End}_{\mathbb{C}}(V)$
	we define $\bm{\ell}(z)=i \varepsilon( \iota(z))$. 
	In other words $\bm{\ell}(z) = \sum_{j=1}^n z_j  f_j$ for $f_j:= i\varepsilon(\iota (e_j))$.
	Since 
	$$  \bm{\ell}(z)^2 = \|z\|^2 \cdot \operatorname{Id}_S$$ we have a proper translation bounded function.
	We call these \emph{Clifford length functions}.
	
	Notice that the translation bounded property is trivially verified by linearity. In particular such length functions can be pulled back: let $s: G_1 \to G_2$ be a map of sets such that $s(g)=0 $ if and only if $g=0$ (then if it is a group morphism it has to be injective) and has finite fibers. Then if $\bm{\ell}$ is a Clifford length function also $\bm{\ell} \circ s$ is a translation bounded proper length function; however it is  not Clifford,
	 in general.
	\demo
\end{exa}

Given a proper translation bounded  function $\bm{\ell}:G \rightarrow \mathbb{B}_h(V)$ on a group $G$ we define an odd spectral triple
\begin{equation}\label{group}
	(\mathbb{C}G, \ell^2 (G) \otimes V , M_{\bm{\ell}}).
\end{equation} The representation is obtained by amplification of the left regular representation $\lambda_g \delta_x=\delta_{gx}$; thus 
$$\lambda(\delta_g) ( \delta_{x} \otimes v):= \delta_{gx} \otimes v$$
 while the Dirac operator is just the multiplication action of $\bm{\ell}$
initially defined on a dense subspace; concretely
\begin{displaymath}
	M_{\bm{\ell}}(\delta_{x} \otimes v) = \delta_{x}\otimes \bm{\ell}(x)v
\end{displaymath}
on $\mathbb{C}G \otimes V$ (algebraic tensor product).
The assumption \eqref{propercondition} in  definition
\ref{generalisedlenght} implies that $(\mathbb{C}G, \ell^2 (G) \otimes V , M_{\bm{\ell}})$ is non-degenerate.

The eigenvectors of the Dirac operator for the spectral triple $(\mathbb{C}G, \ell^2 (G) \otimes V , M_{\bm{\ell}})$ are the vectors $\delta_g\otimes v_j(g)\in \ell^2(G)\otimes V$, where the $v_j(g)$'s are the eigenvectors of $\bm{\ell}(g)$. The family of singular values (with multiplicity) for $M_{\bm{\ell}}$ is then given by 
$\{s_j(g), (g,j)\in G\times \{1,\dots,\dim V\}\}$.

Let us now discuss when the spectral triple $(\mathbb{C}G, \ell^{2}(G)\otimes V, M_{\bm{\ell}})$ is finitely summable. Let us set 
\begin{displaymath}
 B_n=\{g\in G: \min\operatorname{Sp}(|{\bm{\ell}}(g)|)\leq n\}.
\end{displaymath} Note that the translation function $\bm{\ell}$ being proper,  $\#B_n<\infty$. In particular $G$ is at most countable.
\begin{defn}\label{Pgrowth}
	We shall say that $G$ has polynomial growth w.r.t. the function $\bm{\ell}$ if $\#B_n$ grows at most polynomially for $n\to\infty$. In this case we call \emph{growth of $G$} the number
	$$
	d_G=\limsup_n\frac{\log(\# B_n)}{\log n},
	$$
	(cf. e.g \cite{Miln} for the case of the word length for finitely generated groups). 
\end{defn}
For any $g\in G$, we denote by $s_j(g)$, $j=1,\dots \dim V$, the singular values (with multiplicity) of $\bm{\ell}(g)$ and set 
$$
\Sigma_n \coloneqq \{(g,k)\in G\times \{1,\dots,\dim V\}: s_k(g)< n\}.
$$
On the one hand $\Sigma_n\subset B_n\times \{1,\dots,\dim V\}$, therefore $\#\Sigma_n\leq\#B_n\cdot \dim V$.
On the other hand, for any $g\in B_n$ there exists $k\in \{1,\dots,\dim V\}$ such that $(g,k)\in \Sigma_{n+1}$, therefore $\#B_n\leq\#\Sigma_{n+1}$. This implies that
$\displaystyle d_G=\limsup_n\frac{\log(\# \Sigma_n)}{\log n}$.

\begin{prop}\label{growth}
	The  spectral triple $(\mathbb{C}G, \ell^2 (G) \otimes V , M_{\bm{\ell}})$ is finitely summable {\it iff } $G$ has polynomial growth w.r.t. the proper translation bounded function $\bm{\ell}$. In this case, $\textup{abs} (\zeta_D)$ coincides with the growth of $G$.
\end{prop}
\begin{proof}
	By definition, 
	$\zeta_{M_{\bm{\ell}}}(s)=\tr \big((I+M_{\bm{\ell}}^2)^{-s/2}\big)$ and
	$
	\textup{abs}(\zeta_{M_{\bm{\ell}}})=\inf\{s\geq0:\zeta_{M_{\bm{\ell}}}(s)<\infty\}.
	$
%	Recalling definition \ref{defn:inf-order}, we get 
%	$\textup{abs}(\zeta_{M_{\bm{\ell}}})=o((I+M_{\bm{\ell}}^2)^{-1/2})$, 
	Therefore
	the equality 
	\begin{displaymath}
	\textup{abs}(	\zeta_{M_{\bm{\ell}}}) =  
		\displaystyle\limsup_{n}\frac{\log(\lambda_{1/n}((I+M_{\bm{\ell}}^2)^{-1/2}))}{\log n}
	\end{displaymath} follows  by theorem \ref{thm:inf-order}.
	We then observe that 
	\begin{align*}
		\lambda_{1/n}((I+M_{\bm{\ell}}^2)^{-1/2})
		&=
		\#\{k\in\N: \mu_k((I+M_{\bm{\ell}}^2)^{-1/2})> 1/n\}\\
		&=\#\{(g,k)\in G\times \{1,\dots,\dim V\}: \sqrt{1+s_k(g)^{2}}< n\}
	\end{align*}
	and, since $|M_{\bm{\ell}}|\leq (I+M_{\bm{\ell}}^2)^{1/2}\leq I + |M_{\bm{\ell}}|$,
	$$
	\#\Sigma_{n-1}
	\leq \#\{(g,k)\in G\times \{1,\dots,\dim V\}: \sqrt{1+s_k(g)^{2}}\leq n\}
	\leq\#\Sigma_n.
	$$
	Finally,
	$$
	d_G=\limsup_n\frac{\log(\# \Sigma_n)}{\log n}
	=\limsup_{n}\frac{\log(\lambda_{1/n}((I+M_{\bm{\ell}}^2)^{-1/2}))}{\log n}
	=\textup{abs}(\zeta_{M_{\bm{\ell}}}).
	$$
\end{proof}

\subsection{Spectral Triples From Twisted  Actions}\label{sezequicont}

Let $G$ be a discrete group with a twisted action $(\rho,\sigma)$ on a $C^*$-algebra $A$. We assume that a spectral triple $(\mathcal{A},H,D)$ with corresponding representation 
$\pi:A \rightarrow \mathbb{B}(H)$
has been fixed on $A$ and that $\pi$ is faithful.
Here $\mathcal{A} \subset A$ is a smooth dense subalgebra. We give the following definitions generalising the untwisted case \cite{hawkins2013spectral}.

\begin{defn}
We say that the twisted action $(\rho, \sigma)$ is \emph{smooth} if it restricts to a twisted action on $\mathcal{A}$. In other words $\rho: G \rightarrow \operatorname{Aut}(\mathcal{A})$ and $\sigma: G \times G \rightarrow U(\mathcal{A})$. 	
\end{defn}

We consider the following setup for the construction of spectral triples.

\begin{setup}\label{setuptwisted}
Assume we have fixed:
	\begin{enumerate}
		\item a spectral triple $(\mathcal{A}, H,D)$ on the $C^*$-algebra $A$.
		\item A twisted smooth action $(\rho,\sigma)$ on $A$ such that  
		\begin{equation}\label{equicont}
			\sup_{x\in G}\| [D,\pi(\rho_x(a)\sigma_{x,y})]\|<\infty
		\end{equation}
		for every $a\in \mathcal{A}$ and $y\in G$. 
		\item A proper translation bounded length function ${\bm{\ell}}:G \rightarrow \mathbb{B}_h(V)$. 
	\end{enumerate}
\end{setup}

Condition \eqref{equicont} generalizes to the twisted case the one used in \cite{hawkins2013spectral} which, when the triple $(\mathcal{A}, H,D)$ satisfies the Lipschitz condition, corresponds to a metrically \mbox{equicontinuous action}. \\

If the spectral triple $(\mathcal{A} , H,D, \pi)$ is odd,
consider the amplification  $\pi\otimes 1_{\ell^{2}(G,V)}: A \rightarrow \mathbb{B}(H\otimes \ell^{2}(G,V))$ of the representation. Then we have the (left) induced representation $\Pi$ of $A\rtimes_r^{\rho,\sigma}G$ on $H\otimes  \ell^2(G, V)\cong \ell^2(G, V\otimes H)$. From formula \eqref{left} it follows  \begin{displaymath}
	\Pi(a_x\delta_x)(\xi\otimes \delta_y\otimes v  ) = \pi\big{(}\rho_{y^{-1}x^{-1}}(a_x)\sigma_{y^{-1}x^{-1},x}\big{)}\xi\otimes \delta_{xy}\otimes v  
\end{displaymath} for the basic elements  $a_x\delta_x$, 
for $v\in V$ and $\xi \in H$.
We are ready to construct the spectral triple following \cite{hawkins2013spectral}. Notice that if the twisted action is smooth it makes sense to consider the twisted convolution algebra $C_c(G,\mathcal{A})$ dense in the reduced twisted crossed product.

\begin{thm}\label{spectraltripleconstruction}  
Under the assumptions in setup \ref{setuptwisted}:
	\begin{enumerate}
		\item[(1a)] \textup{[Odd to even case]}. 	Assume that the spectral triple on $A$ is odd. Then we have an induced even spectral triple 
		$(C_c(G,\mathcal{A}), {\widetilde{H}}, \widetilde{D})$
		on $A\rtimes^{\rho,\sigma}_r {G}$ with $\widetilde{H}= (H\otimes \ell^2({G})  \otimes V   ) \oplus  (H\otimes\ell^2({G})  \otimes V)$, the representation is the direct sum $\Pi \oplus \Pi$ and the Dirac operator has the form
		\begin{equation}\label{Diractilde}
			\widetilde{D}= \left(\begin{matrix}
				0 &  D\otimes 1_{\ell^2({G})} \otimes 1_V  -i 1_{{H}}\otimes M_{\bm{\ell}} \\
				D\otimes 	1_{\ell^2({G})} \otimes 1_V  +i1_{{H}}\otimes  M_{\bm{\ell}}  & 0
			\end{matrix}\right).
		\end{equation}
		This spectral triple is non-degenerate provided $(\mathcal{A} , H,D, \pi)$ is non-degenerate. 
		
		\item[(1b)] \textup{[Even to odd case]}. 	If the spectral triple on $A$ is even with ${H}={H}^+ \oplus {H}^-$, representation $\pi= \pi^+ \oplus \pi^-$ and $D=\left(\begin{array}{cc}0 &  D^- \\D^+ & 0\end{array}\right)$
		then we can construct an odd spectral triple 
		$ (C_c(G,\mathcal{A})\, \widetilde{{H}}, \widetilde{D})$
		on $A\rtimes^{\rho,\sigma}_{r} G $ with Hilbert space $\widetilde{{H}}= ({H}^+\otimes  \ell^2({G}) \otimes V  ) \oplus ( {H}^-\otimes \ell^2({G}) \otimes V )$ and the representation $\Pi$ is induced by $\pi$. It is diagonal whose components are induced by $\pi^{\pm}$. The Dirac operator is:
		\begin{equation}
			\widetilde{D}= \left(\begin{matrix}
				1_{{H}^+}\otimes M_{\bm{\ell}}  & D^- \otimes 1_{\ell^2({G})}\otimes 1_V    \\ 
				D^+\otimes 1_{\ell^2({G})}\otimes 1_V    & -1_{{H}^-}\otimes M_{\bm{\ell}} 
			\end{matrix}\right).
		\end{equation}
		This spectral triple is non-degenerate if the triple on $A$ is.
		
		\item[(2)] \textup{[Finite summability]}.  	If the triple $(\A,\H, D)$ is finitely summable and $G$ has polynomial growth w.r.t. the translation bounded proper function $\bm\ell$, then also the spectral triple 
		$A \rtimes^{\rho,\sigma}_r G$ is finitely summable and
		\begin{equation}\label{stima2}
			\textup{abs}(\zeta_{\widetilde D})\leq \textup{abs}(\zeta_{D}) + d_G.
		\end{equation} If $\lim_n\frac{\log(\# B_n)}{\log n}$ exists, then the equality in \eqref{stima2} holds true.
\end{enumerate}
\end{thm}
We notice that if the group $G$ is finitely generated and $\bm{\ell}(g)$ is given by the word length, then $d_G$ does not depend on the choice of the generators and the limit $\lim_n\frac{\log(\# B_n)}{\log n}$ exists.

\begin{proof}
	We write the proof for the \textquotedblleft odd to even\textquotedblright \, case, the other one being similar. 
	As in \cite{hawkins2013spectral}, most of the properties will follow because we are following the prescription for the exterior Kasparov product.
	In particular the Dirac operator is a sum $\widetilde{D}=D_1 + iD_2$ with the two blocks
	\begin{equation}\nonumber
		D_1:=\left(\begin{matrix}
			0 & D\otimes 1_{\ell^2({G})} \otimes 1_V  \\
			D\otimes 1_{\ell^2({G})} \otimes 1_V  & 0
		\end{matrix}\right), \quad D_2:=
		\left(\begin{matrix}0 &  -  1_{{H}}\otimes M_{\bm{\ell}} \\
		1_{{H}}\otimes M_{\bm{\ell}}  & 0
		\end{matrix}\right) .
	\end{equation}
	The smooth algebra $C_c(G,\mathcal{A})$ preserves the domain of $\widetilde{D}$ and it allows us to estimate the commutators
	$[D_1,\Pi\oplus \Pi(a)]$ and $[D_2,\Pi\oplus \Pi(a)],$
	for $a\in C_c(G,\mathcal{A})$ separately. Of course we just have to compute these commutators for $a=a_x\delta_x$. It follows by a straightforward computation that
\begin{displaymath}
		[D_1, \Pi\oplus \Pi (a_x\delta_x)]\left(\begin{matrix} \xi\otimes  \delta_y \otimes  v   \\
			  \eta \otimes\delta_{z}  \otimes w
		 \end{matrix}\right)=
	\left(\begin{matrix}
		  \big{[}D,\pi(\rho_{z^{-1}x^{-1}}(a_x)\sigma_{z^{-1}x^{-1},x})\big{]}\eta \otimes\delta_{xz} \otimes w\\ 
		\big{[}D,\pi(\rho_{y^{-1}x^{-1}}(a_x)\sigma_{y^{-1}x^{-1},x}\big{]}\xi \otimes\delta_{xy} \otimes v
	 \end{matrix}\right)
\end{displaymath} which is bounded by condition \eqref{equicont}. 	Concerning the second commutator we have
	\begin{displaymath}
		[D_2,\Pi \oplus \Pi(a_x\delta_x)]=
		\left(\begin{array}{cc}0 & -\big{[}1_{H}\otimes M_{\bm{\ell}},   \Pi(a_x\delta_x)\otimes 1_V \big{]} \\
			\big{[}1_H\otimes M_{\bm{\ell}}, \Pi(a_x\delta_x)\otimes 1_V \big{]} & 0\end{array}\right).
	\end{displaymath}
		Now 
		\begin{displaymath}
			\big{[}1_H\otimes M_{\bm{\ell}}, \Pi(a_x\delta_x) \otimes 1_V\big{]} ( \xi \otimes  \delta_y\otimes v)=  \Pi(a_x\delta_x)(\xi \otimes \delta_y)\otimes (\bm{\ell}(xy)-\bm{\ell}(y))v
		\end{displaymath} which is bounded by the traslation boundedness of the length function.
	Compactness of the resolvent and non degenerateness follow from the corresponding statements for the Kasparov exterior product \cite{baaj1983theorie}.

We finally prove the summability  properties. We shall use the notations and results of Appendix \ref{app1}.
	Let us observe that, if $T$ is a positive invertible operator with compact inverse and if we denote by $N_t(T)$  the number (with multiplicity) of the eigenvalues of $T$ lower than $t$, then we have
	$N_t(T)=\#\{n\geq0: \mu_n(T^{-1})> t\}$ and $o(T^{-1})= \limsup_{t\to\infty}\frac{\log(N_{t}(T))}{\log t}$.
	Therefore 
	$$
	\textup{abs}(\zeta_{\widetilde{D}})=\limsup_{t\to\infty}\frac{\log(N_{t}((I+\widetilde{D}^2)^{1/2}))}{\log t}.
	$$
	Since $\widetilde{D}^2=\big(I\otimes M_{\bm{\ell}}^{2} + D^{2}\otimes I\big)
	\begin{pmatrix}1&0\\0&1\end{pmatrix}$
	and the eigenvalues of $I\otimes M_{\bm{\ell}}^{2} + D^{2}\otimes I$ are $\mu^2+\nu^2$ where $\mu$, resp. $\nu$ is a singular value of $M_{\bm{\ell}}$, resp. $D$, we have
	$$
	N_t((I+\widetilde{D}^2)^{1/2}))=2N_t\big(\sqrt{I\otimes (I+M_{\bm{\ell}}^{2}) + D^{2}\otimes I}\big)
	$$ 
	and 
	\begin{displaymath}
		\begin{split}
			N_{t/\sqrt2}(\sqrt{I+M_{\bm{\ell}}^2}) N_{t/\sqrt2}(\sqrt{I+D^2}) &		\leq
			N_t\big(\sqrt{I\otimes (I+M_{\bm{\ell}}^{2}) + D^{2}\otimes I}\big) \\
			& \leq 		N_{t}(\sqrt{I+M_{\bm{\ell}}^2}) N_{t}(\sqrt{I+D^2}),
		\end{split}
	\end{displaymath}
	hence
	\begin{align*}
		\textup{abs}(\zeta_{\widetilde{D}})&=\limsup_{t\to\infty}\frac{\log(N_{t}((I+\widetilde{D}^2)^{1/2}))}{\log t}
		\\
		&=\limsup_{t\to\infty} \Bigg{(}
		\frac{\log(N_{t}(\sqrt{I+M_{\bm{\ell}}^2}))}{\log t} +
		\frac{\log(N_{t}(\sqrt{I+D^2}))}{\log t} \Bigg{)}
		\leq d_G+\textup{abs}(\zeta_{D})
	\end{align*}
	where the equality holds if any of the limits exists.
\end{proof}
\subsection{The Equivariant Construction}\label{equivariant}In the construction of Subsection \ref{sezequicont}, the representation of the twisted crossed product is obtained by starting from just a representation of the algebra on $H$ (via the induced representation). If we further have a twisted covariant representation on the same Hilbert space, then we can construct a $K$-homologically equivalent spectral triple. This follows remark $2.9$ in \cite{hawkins2013spectral} and is related to the \emph{equivariant} Kasparov product. 

Let $(\pi, U)$ be a twisted covariant representation of $(A,G,\rho,\sigma)$ on a Hilbert space $H$ and suppose that $\pi$ is faithful. Let $V$ be a finite dimensional vector space; this space will play a role only after proposition \ref{45}. Combining with the left regular representation of the group, we get a new covariant representation $(\hat{\pi},\hat{L})$ of $(A,G,\rho,\sigma)$ on $H\otimes \ell^2(G)\otimes V$ as follows:
	\begin{equation}\label{twisted2}
		\begin{dcases}
			\hat{\pi}(a)(\xi\otimes \delta_{g}\otimes v) = \pi(a)\xi\otimes \delta_{g}\otimes v \\
			\hat{L}_{h}(\xi\otimes \delta_{g}\otimes v) = U_{h}\xi\otimes \delta_{hg}\otimes v
		\end{dcases}
	\end{equation}
	for $a\in A$, $\xi\in H$, $v\in V$ and $g,h\in G$. 	It is immediate to check that $(\hat{\pi},\hat{L})$ is twisted covariant. Consider now the unitary operator $W\colon H\otimes \ell^2(G)\otimes V\rightarrow H\otimes \ell^2(G)\otimes V$ defined by 
	\begin{equation}\label{27}
		W(\xi\otimes \delta_{g}\otimes v) = \pi\left(\sigma_{g,g^{-1}}^{*}\right)U_{g}\xi\otimes \delta_{g}\otimes v.
	\end{equation}
Recall the twisted covariant representation $(\widetilde{\pi},\widetilde{L})$ (equations \eqref{leftpair}) whose integrated form is the left regular representation of $A\rtimes_r^{\rho,\sigma}G$. We consider in the following the covariant couple $(\widetilde{\pi}\otimes \operatorname{Id}, \widetilde{L}\otimes \operatorname{Id})$ on $H\otimes \ell^2(G)\otimes V$. With a slight abuse of notation we continue to call it $(\widetilde{\pi},\widetilde{L})$.
\begin{prop}\label{45}The operator $W$ satisfies
		$W\widetilde{\pi}(a)W^{*} = \hat{\pi}(a)$ and 	$W\widetilde{L}_{h}W^{*} = \hat{L}_{h}$. In other words $(\hat{\pi},\hat{L})$ is unitarily equivalent to $(\widetilde{\pi},\widetilde{L})$. Since $\pi$ is injective, the integrated form $\hat{\pi} \times \hat{L}$ descends to a representation of the reduced twisted crossed product.
	\end{prop}
	
	\begin{proof}The first equality is a straightforward computation:
		\begin{displaymath}
			\begin{split}
				W\widetilde{\pi}(a)W^{*}(\xi\otimes \delta_{g}\otimes v) &= \pi(\sigma_{g,g^{-1}}^{*})\left[U_{g}\pi(\rho_{g^{-1}}(a))U_{g}^{*}\right]\pi(\sigma_{g,g^{-1}})\xi \otimes \delta_{g}\otimes v \\
				&= \pi(\rho_{gg^{-1}}(a))\xi\otimes \delta_{g}\otimes v = \pi(a)\xi\otimes \delta_{g}\otimes v\\
				&= \hat{\pi}(a)(\xi\otimes \delta_{g}\otimes v).
			\end{split}
		\end{displaymath}
		We prove the second one by showing that $W \widetilde{L}_h=\hat{L}_hW$. Thus
		$$W\widetilde{L}_h (\xi \otimes \delta_x \otimes v)=\pi(\sigma^*_{hx,x^{-1}h^{-1}})U_{hx}\pi(\sigma_{x^{-1}h^{-1},h})\xi \otimes \delta_{hx}\otimes v$$ and
		$$ \hat{L}_h W(\xi \otimes \delta_x\otimes v)=U_h\pi(\sigma^*_{x,x^{-1}})U_x \xi\otimes \delta_{hx}\otimes v.$$
			It is immediate to check that these two vectors are equal using the fact that $(\pi,U)$ is a twisted covariant representation. Indeed for every $x,y\in G$ we have $\pi(\sigma_{x,y})=U_xU_yU_{xy}^*$. In other words the two vectors are converted in quantities involving only the unitaries $\{U_g,U_{g}^*\}_{g\in G}$ and they coincide. 	The rest of the proof follows from the fact that $\widetilde{\pi}\times \widetilde{L}$ is the defining representation of the reduced crossed product.
\end{proof}
We pass now to describe spectral triples.
\begin{defn}
	Let $(A,G,\rho,\sigma)$ be a twisted dynamical system with $A$ a unital $C^{*}$-algebra. A spectral triple $(\mathcal{A},H,D,\pi)$ on $A$ is \emph{twisted} $G$-\emph{equivariant} or simply $G$-equivariant, if there exists  a unitary map $U\colon G\rightarrow \mathcal{L}(H)$ such that:
	\begin{enumerate}
		\item $(\pi, U)$ is a twisted covariant representation of $(A,G,\rho,\sigma)$ on $H$.
		\item The operators $U_{g}$  leave the domain of $D$ invariant for all $g\in G$.
		\item The commutator $[D,U_{g}]$  extends to a bounded operator on $H$ for every $g\in G$. 
	\end{enumerate}
\end{defn}

Given an odd twisted equivariant spectral triple $(\mathcal{A},H,D,\pi)$,
and a proper translation bounded length function ${\bm{\ell}}:G \rightarrow \mathbb{B}_h(V)$,
 we can define a spectral triple on $A\rtimes_{r}^{\rho, \sigma}G$ as follows: the dense subalgebra is $C_{c}(G,\mathcal{A})$, the Hilbert space is $\widetilde{H} = (H\otimes \ell^{2}(G)\otimes V)\oplus(H\otimes \ell^{2}(G)\otimes V)$ and the representation is two copies of the integrated form of \eqref{twisted2} after being extended trivially on the $V$ factor.
  The Dirac operator is defined as in \eqref{Diractilde}. 
To relate this construction to the triple defined in Subsection \ref{sezequicont}, we begin with the observation that the two representations supporting the triples are unitarily equivalent, the unitary being $\mathcal{W}:=W \oplus W$ with $W$ taken from proposition \ref{45}.
Now the same unitary conjugates $\widetilde{D}$ to an operator
which on simple tensors operates as
$$\mathcal{W}\widetilde{D}\mathcal{W}^{*}\left(\begin{array}{c}
\xi \otimes \delta_g \otimes v
 \\ \eta\otimes \delta_h \otimes w
 \end{array}\right)=\left(\begin{array}{c}
 \pi(\sigma_{h,h^{-1}}^*)U_hDU_h^*\pi(\sigma_{h,h^{-1}})\eta \otimes \delta_h \otimes w-i\eta \otimes \delta_h \otimes \bm{\ell}(h)w
 \\ 
 \pi(\sigma_{g,g^{-1}}^*)U_gDU_g^*\pi(\sigma_{g,g^{-1}})\xi \otimes \delta_g \otimes v+i\xi \otimes \delta_g \otimes \bm{\ell}(g)v
  \end{array}\right).
$$ 
In particular, if $[D,U_{g}]$ is uniformly bounded in $g\in G$, then the twisted action is automatically equicontinuous in the sense of \eqref{equicont}. Furthermore,  as $U_hU_{h^{-1}}=\pi(\sigma_{h,h^{-1}})$ by definition, the family
	\begin{equation}\nonumber
		\big{\{}[D,\pi(\sigma_{h,h^{-1}})]  \big{\}}_{h\in G}
	\end{equation}
	is uniformly bounded and  $\mathcal{W}\widetilde{D}\mathcal{W}^{*}$ is a bounded perturbation of $\widetilde{D}$. Therefore, the two operators lead to the same
 $K$-homology class.  Analogous computations can be performed for the \textquotedblleft even to odd case\textquotedblright. \\

\noindent \emph{Summing up}: starting from an odd (twisted) equivariant spectral triple $(\mathcal{A},H,D,\pi)$ with covariant representation $(\pi,U)$ on $H$, we define a triple
$(C_{c}(G,\mathcal{A}),\widetilde{H},\widetilde{D},\hat{\pi}\times \hat{L})$ on the twisted crossed product
 $A\rtimes_{r}^{\rho, \sigma}G$.
 Here $\widetilde{H} \coloneqq  (H\otimes \ell^{2}(G)\otimes V)\oplus(H\otimes \ell^{2}(G)\otimes V)$, the representation 
 $\hat{\pi}\times \hat{L}:C_c(G,\mathcal{A})\rightarrow \mathbb{B}(\widetilde{H})$ is obtained by doubling the representation
 \begin{equation}\label{repcovarinduced}
 	\hat{\pi}\times \hat{L}(a_h\delta_h)(\xi \otimes \delta_x \otimes v)=\pi(a_h)U_h\xi \otimes \delta_{hx}\otimes v
 \end{equation}
 and the Dirac operator is:
 	\begin{equation}\label{diracequivariant}
		\widetilde{D}= \left(\begin{matrix}
			0 &  D\otimes 1_{\ell^2({G})} \otimes 1_V  -i 1_{{H}}\otimes M_{\bm{\ell}} \\
		  D\otimes 	1_{\ell^2({G})} \otimes 1_V  +i1_{{H}}\otimes  M_{\bm{\ell}}  & 0
	  \end{matrix}\right).
	\end{equation}
If we instead, begin with an even equivariant spectral triple
$(\mathcal{A},H^+\oplus H^-,D)$ with 
representation $\pi= 
\begin{pmatrix}
\pi^+ & 0 \\
0 & \pi^-
\end{pmatrix}$
and operator
$D= 
\begin{pmatrix}
0 & D^- \\
D^+ & 0
\end{pmatrix},$
 on the algebra $\mathcal{A}$, we can, in a similar way construct an odd spectral triple with Hilbert space
 $\widetilde{H} := (H^+ \otimes \ell^{2}({G})\otimes V) \oplus (H^- \otimes \ell^{2}({G})\otimes V)$
 carrying a representation obtained specializing formula \eqref{repcovarinduced} separately to the components $\pi^{\pm}$ 
 and operator
 \begin{equation}
 \widetilde{D} =
\begin{pmatrix}
	I_{H^+} \otimes M_{\vecLength} &  D^- \otimes I_{\ell^2({G})} \otimes I_V \\
	D^+ \otimes I_{\ell^2({G})} \otimes I_V  & -I_{H^-} \otimes  M_{\vecLength}
\end{pmatrix}.
\end{equation}
 
At the end of this section we show that in the equivariant case, the constructed spectral triple can be understood as the restriction of an exterior product of spectral triples. 

Let us specify what we mean by restriction for spectral triples: fix a spectral triple $(\mathcal{A},H,D)$ on the unital $C^*$--algebra $A$; 
let $(\mathcal{B},B)$ be a pair consisting of a unital $C^*$-algebra $B$ with a fixed dense $*$-subalgebra $\mathcal{B}$. 
Here $\mathcal{B}$ is though as being the smooth subalgebra of $B$, for if we have a unital $\ast$-homomorphism $g:B\to A$ mapping $\mathcal{B}$ to $\mathcal{A}$ we say that $g$ is smooth.
In this situation, composing the representation  of $A$ with $g$ we obtain
a
spectral triple $(\mathcal{B},H,D)$. 
\begin{defn}\label{pullback}
We call the spectral triple $(\mathcal{B},H,D)$ the \emph{pullback} spectral triple of $(\mathcal{A},H,D)$ under $g$. If $g$ is a unital inclusion we say that the obtained spectral triple is the \emph{restriction} of $(\mathcal{A},H,D)$.
\end{defn}

We pass now to the construction of a morphism:
$$\phi_A^{r}:A\rtimes_r^{\rho,\sigma}G \longrightarrow A\rtimes^{\rho,\sigma} G \otimes C^*_r(G)$$ such that $\phi_A^r(a\delta_g)=a\delta_g \otimes \delta_g$ (minimal tensor product).
This map is the twisted version of the one used in \cite{Cuntz1983,anto18localised}. Integrating $(\pi,U)$ defines a non degenerate representation   $\pi \times U: A\rtimes^{\rho,\sigma} G\longrightarrow \mathbb{B}(H)$. Let $B\subset \mathbb{B}(H)$ be the (separable) $C^*$-algebra image of $\pi \times U$. 
 We get a non degenerate morphism 
  $$\psi_A:A\rtimes_r^{\rho,\sigma}G \longrightarrow B\otimes C^*_r(G)$$
(again with respect to the minimal tensor product).
  The analytical details will be included in the proof of the next proposition.
 The algebra ${B}$ carries a spectral triple
 $(\mathcal{B},H,D)$ 
  with the same $D$ as $A$ and smooth algebra $\mathcal{B}$ given by finite sums $\sum_{g\in G} a_gU_g$ with $a_g \in \mathcal{A}$; on the other hand, we consider on $C^*_r(G)$ the spectral triple $(C_c(G),\ell^2(G)\otimes V,M_{\bm{\ell},} \lambda)$ constructed with the length function and regular representation.
  \begin{prop}\label{productrestriction}
  The spectral triple 	$(C_{c}(G,\mathcal{A}),\widetilde{H},\widetilde{D},\hat{\pi}\rtimes \hat{L})$ 
   resulting from the equivariant construction
  is the restriction via $\psi_A$ of the  product of the triples $(\mathcal{B},H,D)$ and $(C_c(G),\ell^2(G)\otimes V,M_{\bm{\ell},} \lambda)$.
  \end{prop}
  \begin{proof}
  	We begin by constructing a morphism $\phi_A^{\operatorname{max}}:A\rtimes^{\rho,\sigma} G\rightarrow A\rtimes^{\rho,\sigma}G \otimes C^*_r(G)$
  	(minimal tensor product).
  	This one exists because it is the integrated form of a covariant couple with values in $A\rtimes^{\rho,\sigma}G\otimes C^*_r(G)$; precisely the couple $(j_A,j_G)$ with 
  	\begin{center}
  			\begin{minipage}{0.4\textwidth}
  			\begin{displaymath}
  				\begin{split}
  					j_A: A & \longrightarrow A\rtimes^{\rho,\sigma}G \otimes C^*_r(G), \\
  					a & \longmapsto  a\otimes \operatorname{1}
  				\end{split}
  			\end{displaymath}
  		\end{minipage}
  		\hspace{0.4cm}
  		\begin{minipage}{0.4\textwidth}
  			\begin{displaymath}
  				\begin{split}
  					j_G: G &\longrightarrow U(A\rtimes^{\rho,\sigma}G \otimes C^*_r(G)), \\
  					g&\longmapsto \delta_g \otimes \delta_g.
  				\end{split}
  			\end{displaymath}
  		\end{minipage}
  	\end{center}

  	Indeed we have presented the universal property of the twisted crossed product with respect to covariant representations on Hilbert spaces but we may also use covariant couples with values in the multipliers of any $C^*$-algebra using Hilbert modules
  	\footnote{{In practice we can just represent faithfully $A\rtimes^{\rho,\sigma}G \otimes C^*_r(G)$ on a Hilbert space to make the construction rest on the Hilbert space based definition}}.
  	 To see that $\phi_A^{\operatorname{max}}$ descends to the reduced crossed product, we adapt the argument in \cite{anto18localised}.
  	Fix a faithful representation ${\mu}: A\rtimes^{\rho,\sigma} G \rightarrow \mathbb{B}(H_{\mu})$ and combine it with the left regular representation $\lambda$ to obtain a faithful representation $\mu\otimes \lambda$ of $A\rtimes^{\rho,\sigma}G \otimes C^*_r(G)$ on ${H_{\mu}}\otimes \ell^2(G)$.  Then $\mu$ corresponds to a covariant representation $(\pi^{\mu},U^{\mu})$ on $H_{\mu}$ and 
  	$( \mu \otimes \lambda) \circ \phi_A^{\operatorname{max}}$
  	corresponds to the covariant representation $(\pi^{\mu}\otimes \operatorname{Id}_{\ell^2(G)}\,,U^{\mu}\otimes \lambda_g )$. On the other hand, we consider  the reduced crossed product as being defined by the covariant representation induced by $\pi^{\mu}$ and
  	we call this couple $(\underline{\pi}^{\mu},L^{\mu})$. By \eqref{leftpair} we have  
  $$	\underline{\pi}^{\mu}(a)(\xi \otimes\delta_x)=\pi^{\mu}(\rho_{x^{-1}}(a))\xi\otimes \delta_x, \quad {L}_x^{\mu} (\xi\otimes \delta_y)=\pi^{\mu}(\sigma_{y^{-1}x^{-1},x})\xi \otimes \delta_{xy}.$$
  	
  	We are precisely in the context of proposition \ref{45}. The unitary
     \begin{displaymath}
  \begin{split}
  	 	\Theta:H_{\mu}  \otimes \ell^2(G) &\longrightarrow H_{\mu}\otimes \ell^2(G) \\
  	\xi \otimes \delta_g & \longmapsto \pi^{\mu}(\sigma_{g,g^{-1}}^*)U^{\mu}_g\xi \otimes \delta_g
  \end{split}
   \end{displaymath} satisfies $$\Theta\,\underline{\pi}^{\mu}\,\Theta^*=\pi^{\mu}\otimes \operatorname{Id}_{\ell^2(G)} \quad \textrm{and}\quad \,\,\,
   \Theta \,{L}_g^{\mu} \,\Theta^* =U^{\mu}_g \otimes \lambda_g,   $$ so that $\Theta \,(\underline{\pi}^{\mu} \times L^{\mu})\,\Theta^*= (\mu \otimes \lambda) \circ \phi_A^{\operatorname{max}}$. This means that the representation $(\mu \otimes \lambda) \circ \phi_A^{\operatorname{max}}$ is unitarily equivalent to the left regular one. It follows that $\phi_A^{\operatorname{max}}$ descends to an injection $\phi_A^r: A \rtimes^{\rho,\sigma}_rG \rightarrow A\rtimes^{\rho,\sigma}G \otimes C^*_r(G)$ defined on the reduced crossed product.
   Note that the operator $D$ on $H$ is a Dirac operator for a spectral triple
   $(\mathcal{B},H,D)$
    thanks to the equivariance property of the starting triple making all the commutators $[a_gU_g,D]$ bounded.
To conclude the proof,
    consider the exterior product
    (see Appendix \ref{regularity})
     of the spectral triples $(\mathcal{B},H,D)$ and $(C_c(G),\ell^2(G)\otimes V,M_{\mathbf{\ell}},\lambda)$. By construction this is the spectral triple on the algebra $\mathcal{B}\otimes C_c(G) \subset B\otimes C^*_rG$ represented on the Hilbert space
     $\widetilde{H}=(H\otimes \ell^2(G)\otimes V) \oplus (H\otimes \ell^2(G)\otimes V)$ via  
     $\widetilde{\Pi}:B\otimes C^*_r(G)\longrightarrow \mathbb{B}(\widetilde{H})$ such that 
     \begin{equation}\label{productrepresentation}
     	  \widetilde{\Pi}((\pi \times U)(a_h\delta_h )\otimes \delta_h)=(\pi(a_h)U_h \otimes \lambda_h \otimes 1_V) \oplus (\pi(a_h)U_h \otimes \lambda_h \otimes 1_V)
     \end{equation}
     on $\mathcal{B}\otimes C_c(G)$, 
         and Dirac operator which is exactly the Dirac operator \eqref{diracequivariant} in the equivariant construction.
Formula \eqref{productrepresentation} implies that the diagram
 $$\xymatrix{
 A\rtimes_r^{\rho,\sigma}G\ar[r]^{\psi_A}\ar[dr]_{\widehat{\pi}\times \widehat{L}}&B\otimes C^*_rG \ar[d]^{\widetilde{\Pi}}\\
{}&\mathbb{B}(\widetilde{H}) }$$
 commutes so that 
 the representation in the exterior product restricts under $\psi_A$ to the representation in the equivariant construction.
 Having already observed that the formulas for  the Dirac operator in the equivariant construction and the one in the exterior product coincide, the result follows.\end{proof}

\section{Regularity of the Equivariant Construction}
	
In the equivariant case, by proposition \ref{productrestriction} the spectral triple is a restriction of an exterior product.
This fact can be useful to prove additional properties such as the regularity. 
Basic definitions and informations about regularity of spectral triples are contained in the Appendix \ref{regularity}. Here, before proving the regularity of the equivariant triple, we discuss in which sense the regularity property is well behaved with respect to pullback (then also with respect to restrictions), together with some facts about the regularity of spectral triples associated to group algebras. We begin with the following obvious observation.
\begin{prop}\label{regularpullback}
The pullback of a regular spectral triple, in the sense of definition \ref{pullback} is regular.
\end{prop}
\begin{proof}
In the context of definition \ref{pullback}, let $(\mathcal{B},H,D)$ be the pullback of a regular spectral triple $(\mathcal{A},H,D)$. Then $(\mathcal{B},H,D)$ is regular if and only if both $(\pi\circ g)(\mathcal{B})$ and $[|D|,(\pi \circ g)(\mathcal{B})]$ are contained in the domain 
$\operatorname{Dom}^{\infty}(\delta)$
 (see \ref{regularity}). But this is clear because $g(\mathcal{B})\subset \mathcal{A}.$ 
\end{proof}
Let us pass now to discuss the regularity property of the spectral triples on group algebras constructed with matrix valued length functions.

\begin{exa}
Consider the spectral triple $(\mathbb{C}G, \ell^2 (G) \otimes V , M_{\bm{\ell}})$. In this case the operator $\abs{M_{\bm{\ell}}}$ is the (self-adjoint extension of the)   multiplication operator by  $h\mapsto\abs{\bm{\ell}(h)}$. Given $\delta_{h}\otimes v\in \ell^{2}(G)\otimes V$ we have
		\begin{equation}\label{first}
			\begin{split}
				[\abs{M_{\bm{\ell}}}, \lambda_{g}] (\delta_{h}\otimes v) &= \abs{M_{\bm{\ell}}}(\delta_{gh}\otimes v) - \lambda_{g}\delta_{h} \otimes \abs{\bm{\ell}(h)}v \\
				&=\delta_{gh}\otimes \left(\abs{\bm{\ell}(gh)} - \abs{\bm{\ell}(h)}\right)v.
			\end{split}
		\end{equation}
		It is known that for any couple of square matrices $S,T$ we have 
		\begin{equation}\label{stima}
			\norm{\abs{S} - \abs{T}}\leq C \norm{S-T}
		\end{equation}
		for a suitable constant $C$ which does not depend on $S,T$ but depends on the dimension of the vector space $V$ (see \cite[Corollary 14]{kosaki1992unitarily}) and this is $C=1$ for $\dim V=1$ and proportional to $\log(\dim V)$ for $\dim V\geq 2$. This proves that the commutator $[\abs{M_{\bm{\ell}}}, \lambda_{g}]$ is bounded for any $g\in G$.
		This property, weaker than the regularity, is called {\em{Lipschitz regularity}}. The triple $(\mathbb{C}G, \ell^2 (G) \otimes V , M_{\bm{\ell}})$ is Lipschitz regular for every translation bounded length function $\bm{\ell}\colon G\rightarrow \mathbb{B}_{h}(V)$. 
	
	In general, the left regular representation $\lambda_{g}$ does not lie in the domain of the derivation $\delta^{k}$ for $k\geq 2$ as the elements $\abs{\bm{\ell}(h)}$ do not commute with each other. If this is the case,  then
	\begin{displaymath}
		\delta^{k}(\lambda_{g})(\delta_{h}\otimes v) = \delta_{gh}\otimes \left(\abs{\bm{\ell}(gh)} - \abs{\bm{\ell}(h)}\right)^{k}v
	\end{displaymath}
	and this is bounded by \eqref{stima}. Understanding when the commutator $[M_{\bm{\ell}}, \lambda_{g} ]$ lies in the domain of $\delta^{k}$ for any $k\geq 1$ leads to a similar problem: for  $\delta_{h}\otimes v\in \ell^{2}(G)\otimes V$ we have
	\begin{displaymath}
		\begin{split}
			\delta([M_{\bm{\ell}}, \lambda_{g}]) (\delta_{h}\otimes v)& = \abs{M_{\bm{\ell}}}\delta_{gh}\otimes \left(\bm{\ell}(gh) - \bm{\ell}(h)\right)v - [M_{\bm{\ell}},\lambda_{g}]\abs{M_{\bm{\ell}}}(\delta_{h}\otimes v)\\
			& =\delta_{gh} \otimes \left[\abs{\bm{\ell}(gh)}\left(\bm{\ell}(gh) - \bm{\ell}(h)\right) - \left(\bm{\ell}(gh) - \bm{\ell}(h)\right)\abs{\bm{\ell}(h)}\right]v \\
		\end{split}
	\end{displaymath}
	which is in general not bounded. However if we make the technical assumptions 
	\begin{equation}\label{tech}
	[\bm{\ell}(g),|\bm{\ell}(h)|]=0, \quad \forall g,h \in G\end{equation}
	then
	\begin{displaymath}
		\delta^{k}([M_{\bm{\ell}}, \lambda_{g}]) = \delta_{gh}\otimes \left(\abs{\bm{\ell}(gh)} - \abs{\bm{\ell}(h)}\right)^{k}(\bm{\ell}(gh) - \bm{\ell}(h))v
	\end{displaymath}
	which is bounded by \eqref{stima}. This discussion can be summarized in the following result:
	\begin{prop}\label{23}If the length function $\bm{\ell}$ satisfies \eqref{tech} the spectral triple $(\mathbb{C}G, \ell^{2}(G)\otimes V, M_{\bm{\ell}})$ is regular.
	\end{prop}
\end{exa}

\begin{rmk}
$(a)$	Note that the hypothesis in proposition \ref{23} is satisfied whenever $\bm{\ell}$ is a Clifford length function. 
\\
$(b)$ Estimate \eqref{stima} is true only for operators on finite dimensional vector spaces (see \cite{kato1973continuity}); this constitutes a serious obstruction to the generalization of the previous result to length functions with values in $\mathbb{B}_{h}(H)$ for $H$ infinite-dimensional Hilbert space.
\end{rmk}

\subsection{Regularity for the Equivariant Triple} \label{5.2}In this subsection we discuss the regularity of the triple $(C_{c}(G,\mathcal{A}), \widetilde{H},\widetilde{D},\hat{\pi}\times\hat{L})$ on $A\rtimes^{\rho,\sigma}_r {G}$ defined in \ref{equivariant}. We adopt the notations of section \ref{equivariant}. Recall in particular the spectral triples
\begin{equation}\label{factortriples}
(\mathcal{B},H,D) \quad \textrm{and}  \quad (C_c(G),\ell^2(G)\otimes V,M_{\bm{\ell},} \lambda)
\end{equation}
 where $\mathcal{B}$ is the algebra of finite sums $\sum_{g\in G} a_gU_g$ with $a_g \in \mathcal{A}$. The completion of $\mathcal{B}$ is the algebra $B\subset \mathbb{B}(H)$ image of the full crossed product by the integrated form of the covariant representation $(\pi,U)$. 
By proposition \ref{productrestriction} the equivariant triple is the restriction of  the product of the triples \eqref{factortriples}. Since the notion of regularity  is  well behaved with respect to the operation of restriction,
identifying sufficient conditions for the regularity of the two factor triples will give sufficient conditions for the regularity of our triple.

\begin{thm}\label{maintheorem}
	Let $(\mathcal{A}, H,D,\pi,U)$ be a $G$-equivariant odd spectral triple over a unital $C^{*}$-algebra $A$ and ${\bm{\ell}}:G \rightarrow \mathbb{B}_h(V)$ a proper translation bounded length function. \\
	
\noindent $(1)$.	If the spectral triples
$(\mathcal{B}, H,D)$ and
			$({C}_c(G), \ell^2 (G) \otimes V , M_{\bm{\ell}})$	 are regular,
 then the triple 
$(C_{c}(G,\mathcal{A}), \widetilde{H},\widetilde{D},\hat{\pi}\rtimes\hat{L})$ on $A\rtimes^{\rho,\sigma}_r {G}$ of the equivariant construction is regular. \\

\noindent $(2).$ Let $(\mathcal{A},H,D)$ be regular and let the group act by order zero operators, i.e. $U_g \in \operatorname{Op}^0(\Delta)$ for every $g\in G$ (the Laplacian being   relative to $D$ on $H$).  
Assume further that 
\begin{itemize}
\item	$[D,U_g]\in \operatorname{Op}^0(\Delta)$ for every $g\in G$,
\item for every integer $k$ and $g\in G$:
\begin{equation}\label{compatibilitywithlaplacian}
[\underbrace{\Delta,[\Delta, [\Delta, \dots [\Delta}_{k\textrm{\, times}},U_g  ]\,]\,]\in \operatorname{Op}^k(\Delta) \quad 	\textrm{and}\quad  [\underbrace{\Delta,[\Delta, [\Delta\dots}_{k\textrm{\, times}}\,,  [D,U_g  ]\,]\,]\in \operatorname{Op}^k(\Delta).\end{equation}
\end{itemize}
Then the triple $(\mathcal{B},H,D)$ is regular. \\

\noindent $(3)$ 
the converse of point $(2)$ holds. 
If $(\mathcal{B},H,D)$ is regular then $(\mathcal{A},H,D)$ is regular and  the group action satisfies:
$U_g \in \operatorname{Op}^0(\Delta)$,  $[D,U_g]\in \operatorname{Op}^0(\Delta)$ and \eqref{compatibilitywithlaplacian} for every integer $k$.
\end{thm}
\begin{proof}
The first statement is clear by proposition \ref{regularpullback}. Concerning the second one, assume that the triple on $A$ is regular and denote with $\mathcal{E}_A$ its canonical algebra of $GDO$. 
Assuming $U_g$ of order zero implies that  $\mathcal{B}W^{\infty} \subset W^{\infty}$ and we can construct the canonical algebra of operators relative to $B$ as explained right before theorem \ref{thm2}. 
We call it $\mathcal{E}_B$. Now 
the remaining assumptions imply that $\mathcal{E}_B \supset \mathcal{E}_A$
is the canonical algebra of $GDO$ that, according to Higson's theorem \ref{thm2} witnesses that the triple over $\mathcal{B}$ is regular. \\
 Let's discuss point $(3)$. For every covariant representation $U_e=\operatorname{Id}$ so that we have an inclusion $\mathcal{A}\subset \mathcal{B}$ and the triple on $\mathcal{A}$ is the restriction of the one on $\mathcal{B}$. It follows that $(\mathcal{A},H,D)$ is regular. Moreover since $\mathcal{A}$ is unital we get all the stated conditions on the $U_g$'s just by direct examination of its canonical $GDO$ algebra. 
 \end{proof}

 \begin{exa}
 Let $\Gamma$ be a discrete group of diffeomorphisms of a smooth (compact) manifold preserving a triangular structure in the sense of \cite[Section 2]{conneslocalindexformula}. Then the spectral triple constructed by Connes and Moscovici satisfies the assumptions in  
 $(2)$ of the above theorem because in their case any $U_g$ and any $[D,U_g]$ belong to $\Dom^{\infty}(\delta)$ (see \cite[Theorem 1.1]{conneslocalindexformula}).
\end{exa}

Let us now briefly discuss the nature of 
the condition 
$[D,U_g]\in \operatorname{Op}^0(\Delta)$ of point $(2)$ in theorem \ref{maintheorem} in the manifold case. Intuitively, it says that the group is acting in an isometric fashion. Indeed, if we consider the case of an odd  Spin manifold, we can show the following fact, which is certainly well known, but it is proved here  for completeness.

 \begin{prop}
Let $M$ be a spin (Riemannian) compact odd dimensional manifold with spinor bundle $S$ and let $U:M \rightarrow M$ be a diffeomorphism that is covered by a fibrewise linear map $\widetilde{U}:S \rightarrow S$ which is unitary on the fibers. This means that we have a commutative diagram 
$$\xymatrix{S\ar[r]^{\widetilde{U}}\ar[d] & S\ar[d] \\
M\ar[r]^U & M.
}	$$
Let $\underbar{U}:\Gamma(M,S) \longrightarrow \Gamma(M,S)$
be the induced operator on sections: 
\begin{displaymath}
	(\underline{U}s)(x):=(\widetilde{U}_{x})^*s(U(x))
\end{displaymath} for $s\in \Gamma(M,S)$ and $x\in M$. 
Denote with $D$ the Dirac operator associated to $M$.
We have $\underline{U}\in \operatorname{Op}^0(\Delta)$
 and it follows that
 $[D,\underline{U}] \in \operatorname{Op}^0(\Delta)$ if and only if $U$ is an isometry.\end{prop}

 \begin{proof}

Since $M$ is compact, $\underline{U}$ is order zero (it does not involve derivatives) and invertible. It follows that 
$[D,\underline{U}]$ is order zero
 if and only if the operator
 $A:=\underline{U}^*[D,\underline{U}]=\underline{U}^*D\underline{U}-D$ is order zero. Though $\underline{U}$ is not a differential operator (it is non local), the operator $A$ is a differential operator of order no more than one so it will be order zero if and only if its principal symbol (of order one) vanishes. 
 More precisely let's check that 
 $\underline{U}^*D\underline{U}$ is differential of order one.
 Following \cite{Crainic}, we show that for every fixed section $s\in \Gamma(M,S)$ and $\varphi \in \Gamma(M,S^*)$ the operator
 $$C^{\infty}(M) \ni f\longmapsto \varphi(\underline{U}^*D\underline{U}(fs))\in C^{\infty}(M)$$
 is a differential operator.
 Since $\underline{U}(fs)=(f\circ U) \underline{U}s$, we can compute
  \begin{equation}\label{leibnitz}
  \varphi(\underline{U}^*D\underline{U}(fs))
  =\varphi(\underline{U}^*c(d(f\circ U))\underline{U}s)+ \varphi(f\underline{U}^*D\underline{U}s).
  \end{equation}
  Here for a function $f$ we denote with $c(df)$ the corresponding Clifford multiplication; indeed we have used the property of the Dirac operator: $D(fs)=c(df)s+fDs$.
  
From Equation \eqref{leibnitz} we have that $[A,f]$ is $C^{\infty}(M)$-linear and so $A$ is a differential operator. The same equation
 can be used to compute the principal symbol of $A$: let $s$ be a section of the spinor bundle and fix a point $x=U(y)\in M$. Then
 	\begin{displaymath}
 \begin{split}
 		\sigma^{1}(A)(d_xf)s_x =i[A,f]s \,\rest{x} &=i\underline{U}^*c(d(f\circ U))\underline{U}s-ic(df)s\, \rest{x} \\
 	&=i\widetilde{U}_{y}\,c(d_yU^t(d_xf)) \, (\widetilde{U}_{y})^*s_x-ic(d_xf) s_x.
 \end{split}
 	\end{displaymath}
 	The thesis follows quickly because this is zero if and only if we have a commutative diagram
 	\begin{equation}\label{intertwiner}\xymatrix{S_y\ar[d]_{c_y(d_yU^t\,(d_xf) ) }
 	\ar[r]^{\widetilde{U}_y} & S_{U(y)\ar[d]^{c(d_xf)}}   
 \\
 S_y\ar[r]_{\widetilde{U}_y} & S_{U(y)}	
 	}
 	\end{equation}
  for $d_yU^t$ the transpose of $dU$. Let us now show that this means that $dU^t$ is unitary: given a function $f$ such that $\|df\|=1$, we have that the linear map $c(df)$ is unitary. Since $\widetilde{U}$ is an isometry, we deduce that $c_y(d_yU^t\,(d_xf))$ is an isometry so that $\norm{d_yU^t\,(d_xf)}=1$. Then $dU^{t}$ preserves the norms. 
   \end{proof}

	\section{Finite Group Actions and Noncommutative Coverings} 
	In this section we introduce regular  noncommutative finite coverings with abelian structure group and we show that they are isomorphic to twisted crossed products with respect to the dual group,
	a result that apparently appeared in \cite{wagner2015noncommutative} for the first time. We also discuss in detail the meaning of the notion that was called regularity in \cite{aiello2017spectral} by comparing it with the familiar one of a free action.
	We refer to \cite{phillips2009freeness} and \cite{schwieger2015free} for further properties of group actions on $C^*$-algebras. \\
	
	In this section we shall always consider the action of a finite abelian group $G$ on the $C^*$-algebra $B$ with fixed point algebra $B^G$. 
	\begin{defn} \label{def-reg-cov}
		A \emph{finite (noncommutative) covering} with abelian group is an inclusion of  unital $C^*$-algebras $A\subset B$ together with an action of a finite abelian group $G$ on $B$ such that $A=B^G$. We say that $B$ is a covering of $A$ with \emph{deck transformation group} \mbox{given by $G$} and denote this structure by $G \curvearrowright B \supset A$.   
	\end{defn}

	\noindent For a finite noncommutative covering, the $G$-action (denoted by $g\cdot b=\gamma_g(b)$) decomposes $B$ in its closed \emph{spectral subspaces}
	\begin{displaymath}
		B_k:=\{b\in B \textrm{ s.t. } \gamma_g(b)=\langle k,g \rangle b \quad \forall g\in G \}, \quad  k\in \widehat{G}
	\end{displaymath}
	where $\widehat{G}$ is the Pontryagin dual group of $G$. Every $B_k$ is a Hilbert bimodule over $A$ with scalar right product $\langle b_1,b_2\rangle:=b_1^*b_2.$ and left product $\langle b_1,b_2\rangle:=b_1b_2^*$. 
	Notice that the induced norm on every $B_{k}$ coincides with the norm of $B$ so that every spectral subspace is complete from the beginning.
	The following relations hold (cf. e.g.  \cite{aiello2017spectral}).
	\begin{prop} \label{prop-1.1}
		With the above notation we have that:
		\begin{enumerate}
			\item[$(1)$] $B_hB_k\subset B_{h+k}$.			
			\item[$(2)$] if $b_k\in B_k$ is invertible, then $b_k^{-1}\in B_{-k}.$ Moreover $b_k^*\in B_{-k}$ for every $b_k \in B_k$.
			
			\item[$(3)$] Each $b\in B$ may be written as $\sum_{k\in \widehat{G}} b_k$ with $b_k\in B_k$.
		\end{enumerate}
	\end{prop}

	Let us now consider  the ordinary crossed product $B\rtimes G$. We can construct a canonical $B\rtimes G - A$ - Hilbert bimodule ${}_{B \rtimes G\,} \mathcal{E}_A$ in the following way:  consider $B$ endowed with
	\begin{itemize}
		\item left action $b_g\delta_g \cdot x:=b_g\gamma_g(x)$ and left inner product 
		\begin{displaymath}
			{}_{B\rtimes G}\langle b_1,b_2 \rangle:= \sum_{g\in G}b_1 \gamma_g(b_2^*)\delta_g,
		\end{displaymath}
		\item right action $b \cdot a:=ba$ and right inner product 
		\begin{displaymath}
			\langle b_1,b_2 \rangle_A:= \frac{1}{|G|} \sum_{g\in G}\gamma_{g}(b_1^*b_2).
		\end{displaymath}
	\end{itemize}
	Then ${}_{B \rtimes G\,} \mathcal{E}_A$ is the bimodule obtained by completing $B$ in the usual way. It turns out that ${}_{B \rtimes G\,} \mathcal{E}_A$ is almost a Morita equivalence bimodule for it may lack the fullness of the left product ${}_{B\rtimes G\,}\langle \cdot,\cdot \rangle$.
	
	\begin{defn} 
The group action is {\em{free}} or equivalently {\em{saturated}}, as called by Rieffel (see \cite[theorem 1.5 and definition 1.6]{MR1103376} and \cite{MR911880}), when ${}_{B\rtimes G\,}\mathcal{E}_{A}$ is a Morita equivalence bimodule.
	\end{defn}
	When $B=C(X)$ with $X$ Hausdorff and compact a theorem of Rieffel 
	(see \cite{MR911880} proposition 7.1.12 and theorem 7.2.6)
	shows that the action $\gamma: G \rightarrow \operatorname{Aut}(C(X))$ is free if and only if the corresponding action $X \curvearrowleft G$ is free in the familiar sense.
	\begin{defn}
		The canonical map $\operatorname{can}: B\odot B \rightarrow B\otimes C(G)$ defined on the algebraic tensor product is the $B-A$-module map given by 
		\begin{equation}
			\operatorname{can} (x \odot y):=\sum_{g\in G} x \gamma_g(y)\otimes \delta_g.
		\end{equation} We say that the $G$-action satisfies the {\em{Elwood condition}} if the map $\operatorname{can}$ has dense range in $B\otimes C(G)$ with respect to the tensor product $C^*$-norm.
	\end{defn}
	
	\begin{thm}[see \cite{phillips2009freeness,schwieger2015free}]\label{characterizationfreeness}Let $G \curvearrowright B \supset A$ be a finite covering with abelian group. The following conditions are equivalent 
		\begin{enumerate}
			\item The action is free.
			\item The action satisfies the Elwood condition.
			\item Every spectral subspace $B_k$ for $k\in \widehat{G}$ is a $A-A$ Morita equivalence bimodule.
			\item The multiplication map induces an isomorphism $B_{-k} \otimes_A B_k \cong A$ of Hilbert $A-A$-modules for every $k \in \widehat{G}$. Here $B_{-k} \otimes_A B_k$ is the balanced tensor product and the image of the multiplication map from $B_{-k}\otimes_A B_k$ coincides with $\overline{B_{-k}B_k}$, the closure in $B$ of the linear span of all the products $b_1b_2$ with $b_1 \in B_{-k}$ and $b_2 \in B_k$.
		\end{enumerate}	
		Moreover if the action is free any $B_k$ is finitely generated and projective as a right $A$-module.
	\end{thm}

We say that an element $x$ in a  right $A$-module $X$ is a generator if 
$X=\{x a, a\in A\}$. Analogously, an element $x$ in a  left $A$-module $X$ is a generator if $X=\{ax, a\in A\}$.
\begin{lemma}\label{right-generator}
Let $A$, $B$ and $G$ as above. Then
\begin{enumerate}
\item If the action is free and, for any $k\in\widehat{G}$, $B_k$  has a generator $\mu_k$ as a right $A$-module, then any $\mu_k$ has a right-inverse in $B_{-k}$.
\item If the action is free and, for any $k\in\widehat{G}$, $B_k$  has a generator $\mu_k$ as a left $A$-module, then any $\mu_k$ has a left-inverse in $B_{-k}$.
\end{enumerate}
\end{lemma}
\begin{proof}
$(1)$: 
Since  $\mu_k$ is a generator for $B_k$ as a right $A$-module, any element in $B_{-k}\otimes_A B_k$ has a representative in $B_{-k}\otimes B_k$ of the form $\mu_{-k}\otimes \mu_k c$, $c\in A$, even though $c$ is not unique in principle. 
Since  $p:\mu_{-k}\otimes \mu_k c\mapsto\mu_{-k} \mu_k c$ is surjective,
 there exists $c\in A$ such that $\mu_{-k}\mu_{k}c=1$, namely $\mu_{k}c\in B_k$ is a right inverse of $\mu_{-k}$. This shows that each $\mu_k$  has a right-inverse in $B_{-k}$. (2): This is similar to (1).
\end{proof}
The following Lemma is probably well-known, we include the proof for the sake of completeness.
\begin{lemma}\label{poldec}
Let $C$ be a $C^*$-algebra, $c\in C$ a left-invertible element. Then $c$ has a polar decomposition $c=vh$ with $v,h\in C$, $h$ an invertible positive element and $v$ isometric.
\end{lemma}
\begin{proof}
Let us identify $C$ with an algebra of operators acting on a Hilbert space $H$, and let $c=vh$ be the polar decomposition in $B(H)$. Since $c$ has a left inverse, it is injective, hence $v$ is an isometry. If $b\in C$ is a left inverse, then $bvh=I$, hence $h$ has a left inverse. Being self-adjoint, it is indeed invertible, therefore $v=ch^{-1}$ is in $C$ 
\end{proof}
\begin{thm}\label{rank1thm}
The following are equivalent:
\begin{enumerate}
\item For any $k\in\widehat{G}$, $B_k$ contains an element which is unitary in $B$.
\item For any $k\in\widehat{G}$, $B_k$ contains an element which is invertible in $B$.
\item For any $k\in\widehat{G}$, $B_k$ is a free, rank-1, right $A$-module.
\item For any $k\in\widehat{G}$, $B_k$ is a free, rank-1, left $A$-module.
\item For any $k\in\widehat{G}$, $B_k$ contains an element which is a generator both as a right and left $A$-module, and the action is free.
\end{enumerate}
\end{thm}
\begin{proof}
$(1)\Rightarrow (2)$ is obvious.
%$(2)\Rightarrow (1)$: if $\mu_k\in B_k$ is invertible in $B$, $a_k=\sqrt{\mu_k^*\mu_k}$ is invertible in $A$, hence $u_k=\mu_k a_k^{-1}\in B_k$ is invertible in $B$. Since $\mu_k=u_ka_k$ is the polar decomposition, $u_k$ is unitary. 
%and is clearly a generator for $B_k$ both as a right and as a left $A$-module.
\\
$(2)\Rightarrow (3)$: the map $a\in A\mapsto \mu_k a\in B_k$ is a right $A$-module isomorphism. As for the freeness, we have to show that, for any $k\in\widehat{G}$, the $A$-bimodule map
$$
\begin{matrix}
p: &B_{-k}&\otimes_A &B_k&\to &A\\
&b_{-k}&\otimes &b_k&\mapsto &b_{-k}b_k
\end{matrix}
$$
 is a bijection. 
Such a map is always injective so we just have to prove surjectivity.
 %Since the  $\mu_k$ are invertible, for any element $x\in B_{-k}\otimes_A B_k$ there is a unique element $a\in A$ such that  $\mu_{-k}\otimes \mu_k a\in B_{-k}\otimes B_k$ is a representative of $x$, and $p(\mu_{-k}\otimes \mu_k a) = \mu_{-k} \mu_k a$ is clearly injective. 
 Setting $a=\mu_k^{-1}\mu_{-k}^{-1}$, we have $a\in A$ and $p(\mu_{-k}\otimes \mu_k a) = 1$, which implies that $p$ is surjective.
\\
$(3)\Rightarrow (1)$: if $\Phi:A\to B_k$ is a right $A$-module isomorphism, namely $\Phi(a_1)a_2=\Phi(a_1 a_2)$, $\mu_k:=\Phi(1)$ is clearly a generator of $B_k$ as a right $A$-module.
Then, according to Lemma \ref{right-generator}, $\mu_k$ has a right-inverse, hence $\mu_k^*$ has a left inverse in $B_k$ and, by Lemma \ref{poldec}, $\mu_k=hv^*$ with $v,h\in B$, $v$ isometric and $h$ invertible. Indeed, $h\in A$ and $v_k^*\in B_k$.
We observe that $I-vv^*$ is in $A$, and $\Phi(I-vv^*)=\mu_k(I-vv^*)=hv^*(I-vv^*)=0$. Since $\Phi$ is an $A$-module isomorphism, we get $I-vv^*=0$, namely $v$ and $v^*$ are unitary.
\\
By the arguments above, $(3)\Leftrightarrow (4)$ and $(2)\Rightarrow (5)$ are obvious.
\\
$(5)\Rightarrow (2)$: Let $\mu_k\in B_k$ be a generator of $B_k$ both as a right and left $A$-module. By Lemma \ref{right-generator}, $\mu_k$ has a right-inverse and a left-inverse, namely it is invertible.
\end{proof}

	\begin{defn}\label{unitaries}
		We say that the finite covering $G \curvearrowright B \supset A$  is {\emph{rank-1 regular}}  if one of the equivalent conditions of theorem \ref{rank1thm} is satisfied. In this case, the action is said to be \emph{rank-1 free}. Any map $\mu:\widehat{G} \longmapsto U(B)$ with ${\mu}_j \in B_j$ and $\mu_e=I$  is called a \emph{frame}.
	\end{defn}

	\begin{rmk}\label{regcov}
		\item{$(i)$} 	
Condition (1) in theorem \ref{rank1thm} has been called regularity in 
\cite{aiello2017spectral} and it has been proved to imply freeness by making use of the Elwood  condition.

		\item{$(ii)$} Rank-1 freeness also implies that the action is faithful. Indeed, if there exists a nontrivial $g\in G$ acting trivially, we may find $k\in\widehat G$ such that $\langle k,g\rangle\ne1$. Therefore the equation $\gamma_g(b)=\langle k,g \rangle b$ is satisfied only for $b=0$ and $B_k$ does not contain any invertible element.
		\item{$(iii)$} As $\mu_{k}\in B_{k}$, we have that $\mu_{k}A\mu_{k}^{*} = A$ for any $k\in \widehat{G}$. 
	\end{rmk}

	\begin{exa}\label{exareg2}Let $G$ be a finite abelian group and $A$ a unital $C^{*}$-algebra. Suppose that $\widehat{G}$   has a twisted action $(\rho,\sigma)$ on $A$. Then the twisted crossed product $B:= A\rtimes^{\rho,\sigma}\widehat{G}$, which contains $A$ in the obvious way, is a rank-$1$ free noncommutative covering with deck transformation group $G$.	It is clear in fact that the fixed point algebra of $B$ under the (bi)dual action  $U_{x}( a_j{\delta}_j)=\overline{\langle j,x \rangle} a_j \delta_j$ of $G$ for $x \in G$ and $j\in \widehat{G}$ is precisely $A$. 
		Moreover the spectral subspaces are $B_k=A\delta_{-k} = \delta_{k}A$ for every $k\in \widehat{G}$ (of course every $\delta_k$ is unitary).
		\demo
	\end{exa}

	In general, the rank-1 regularity assumption is not always satisfied as shown by \cite[Example 1.7]{aiello2017spectral}. 
We give below another counterexample, showing in particular that the properties in Lemma \ref{right-generator} do not imply rank-$1$ regularity. 
This is, of course, a purely noncommutative phenomenon. 
Later we shall give also an example of a commutative covering with free action which is not rank-$1$ regular.
\begin{exa}\label{Toeplitz}
We consider here a free $\Z_2$-action on the Toeplitz algebra $\Tt$ such that the spectral subspace $B_-$ has a generator as a right $A$-module, but contains no invertible elements. Standard references for the basic properties of the Toeplitz algebra are \cite{MR1402012} and \cite{MR1865513}.
\\
Let us define the self-adjont unitary operator $u$ on the Hardy space $H^2$ that, on the orthonormal basis $z^n$, $n\geq0$, acts as $uz^n=(-1)^nz^n$, and denote by $\sg$ the action on $\Tt$ given by 
$\sg(b)=ubu$. This defines a  $\Z_2$-action on $\Tt$. Denoting by $T_f$ the Toeplitz operator associated with the function $f\in C(\bT,\C)$,
we get $\sg(T_z)=- T_z$ therefore, by known properties of Toeplitz operators, $\sg(T_{z^n})=(-1)^nT_{z^n}$ for any $n\in\Z$. 
Setting $(s(f))(z)=f(-z)$ for any function in $C(\bT,\C)$, we deduce 
$\sg(T_f)=T_{s(f)}$. Now if $b\in B_+$, $b=T_f+K$ with $K$ a compact operator, we have $b=\sg(b)=\sg(T_f)+uKu=T_{s(f)}+uKu$. By the uniqueness of the decomposition, we obtain $\sg(T_f)=T_f$. An analogous result holds for $b\in B_-$.
\\
Moreover, since $T_z$ is an isometry, any element in $b\in B_-$ can be written as $(bT_z^*)T_z$ with $bT_z^*\in B_+$, namely $T_z$ generates $B_-$ as a left $A$-module. Analogously, $T_z^*$ generates $B_-$ as a right $A$-module.
\\
We now prove that, 
for any Fredholm, Toeplitz operator $T_f\in B_+$, $\Ind(T_f)$ is even,
and, 
for any Fredholm, Toeplitz operator $T_f\in B_-$, $\Ind(T_f)$ is odd.
Indeed,  $\sg(T_f)=T_f$ implies $s(f)=f$. As a consequence, setting $\f(t)=f(e^{it})$, 
we have $\f(t+\pi)=\f(t)$, namely $\f$ follows the same path twice, hence $f$ has an even winding number and $T_f$ has an even index. 
If $\sg(T_f)=-T_f$ then $T_f=T_z^*T_{zf}$ with $s(zf)=zf$, hence $\Ind(T_f)=\Ind(T_z^*)+\Ind(T_{zf})$ is odd.
\\
Finally, if $b=T_f+K$ is invertible, it is Fredholm with index 0, hence $T_f$ has index zero, which imply that $b\not\in B_-$.
\\
As for the freeness of the action, we observe that, for any $b\in B_+$, $a=T_z^2b$ is in $B_+$, hence $p(T_z^*\otimes T_z^*a)=b$, that is $p$ is surjective
\demo
\end{exa}
	
	Let's examine the commutative case. We consider a compact space $X$ with right action $X \curvearrowleft G$ and quotient $Y=X/G$. The induced action on the functions corresponds to 
	\begin{displaymath}
		(\gamma_g f )(x)=f(x\cdot g),
	\end{displaymath}
	for $g\in G$ and $f\in C(X)$, and the spectral subspaces read as
	\begin{displaymath}
		C(X)_j=\{f \in C(X): \gamma_g(f)=\langle j, g \rangle f  \}, \quad j \in \widehat{G}.
	\end{displaymath}
	Recall that a covering of topological space is regular when its group of automorphisms acts transitively on every fiber. To help the comparison with the algebraic case, we shall call it \emph{topologically regular.}
	Assume that $X$ is Hausdorff, connected and locally path connected, then 
	the $G$ action is free if and only if $X \to Y$ is topologically regular with $G$ being its group of automorphisms. Indeed if the action is free, the Hausdorff assumption combined with the finiteness of $G$ implies that the action is a \emph{covering space action} and we may apply a standard result in basic topology (see for instance the Covering Space Quotient Theorem 12.14 in \cite{Lee2011Introduction}).
	
	In this case, i.e when $ X \to Y$ is topologically regular 
	we can identify $C(X)_j$ with the $C(Y)$-module of sections of a unitary complex line bundle $V_j \to Y$ associated to the character $j:G \to U(1)$. The total space of $V_j$ is the quotient space $(X \times \mathbb{C})/{G}$ with respect the the right diagonal action $(x,z) \cdot g:= (x \cdot g, \langle j, g \rangle \, z)$ and the canonical identification
	\begin{displaymath}
		C(X)_j \cong \Gamma(Y,V_j),
	\end{displaymath}
	associates to the equivariant function $f:X \to \C$ the section $F:Y\to V_j$ defined by $F([x])=[x,f(x)]$ for every $[x]\in Y$. The following facts are well known but it is interesting to specialize them to the rank-$1$ case. 
	\begin{prop}Let $X$ be compact Hausdorff, connected and locally path connected.
		Then the covering $C(Y)\subset C(X) \curvearrowleft G$ is rank-$1$ regular if and only if it is topologically regular and 
		all the		
		associated bundles $V_{j}$ are  trivial: $V_{j} \simeq Y \times \mathbb{C}$ (the isomorphism is not required to preserve the flat structure).
			
	\noindent 	If all the spaces are reasonably good (say $CW$-complexes) then the triviality of the associated bundles is equivalent to the vanishing of the first Chern class $[c_1(V_{j})] \in H^2(Y;\mathbb{Z})$.
	\end{prop}
	\begin{proof}
		Assume first that we have a rank-$1$ regular covering with invertibles $\mu_{j}$, $j=1,...,n$. To prove that the covering is topologically regular we just have to prove that the action is free in the topological sense. This is immediate because $x \cdot g=x$ implies 
		$$\langle j,g\rangle  \mu_j(x)=(\gamma_g\mu_j)(x)=\mu_j( x\cdot g)=\mu_j(x)$$ for every $j$. This is impossible unless $g=e.$ The viceversa is clear once we observe that the topologically regular assumption makes the construction of the vector bundles $V_j$ possible. Any invertible $\mu_j$ corresponds to a trivialization. About Chern classes we refer to \cite{MR1249482}.
	\end{proof}
	
	Notice that if $X$ and $Y$ are manifolds, by Chern-Weil theory since our bundles are flat, the real Chern classes $[c_1(V_j)]\in H^2(Y;\mathbb{R})$ vanish for every $j \in \widehat{G}$. In particular the integer classes $[c_1(V_j)] \in H^2(Y;\mathbb{Z})$ are torsion classes. It follows that  if  $H^2(Y;\mathbb{Z})$ is torsion free then 
	the regularity assumption \eqref{unitaries} is satisfied. In general this is not the case as shown by the following example from \cite{kamber1967flat}.
	\begin{exa}
		Let $\mathbb{Z}_2$ act on the 2-sphere as the universal covering of the projective space $\mathbb{P}^2(\mathbb{R})$. Then the character $j: \mathbb{Z}_2 \rightarrow SO(2)$ mapping the generator to the antipodal produces a flat $SO(2)$-bundle $V_{j}$ which is non trivial because its Euler class generates $H^*(\mathbb{P}^2(\mathbb{R});\mathbb{Z})$.
		Notice that
		under the isomorphism $SO(2) \cong U(1)$ we see that $V_j$ has a natural complex structure. In this way the first Chern class corresponds to the Euler class. \demo
	\end{exa}
	
	In example \ref{exareg2} we have shown that $A\subseteq A\rtimes^{\rho,\sigma}\widehat{G}$ is a rank-$1$ regular finite noncommutative covering. The converse holds true: any finite rank-1 regular covering is a twisted crossed product. This was proved by Wagner in \cite{wagner2015noncommutative}, Appendix A, with a slightly different terminology.
	
\begin{thm}\label{isocrossed}
		Let $G \curvearrowright B \supset A$	be a finite rank-1 regular noncommutative covering with $G$ abelian. Any frame $\mu:\widehat{G} \longmapsto U(B)$ determines a twisting pair $(\rho,\sigma)$ by
		\begin{displaymath}
			{\rho}_j(a)={\mu}_ja{\mu}_j^*, \quad \textrm{and}\quad   \sigma_{j,k}={\mu_j}\mu_k\,\mu_{j+k}^*, \quad j,k \in \widehat{G}
		\end{displaymath}
		together with an isomorphism $\Phi:  A\rtimes^{\rho,\sigma} \widehat{G}\rightarrow B$  which on the generators is given by 
		\begin{displaymath}
			\Phi(a_j\delta_j)=a_j\mu_j.
		\end{displaymath}
		This isomorphism is equivariant with respect to the natural (dual) action of 
		$G$ on $A \rtimes^{\rho,\sigma} \widehat{G}$.
\end{thm}
	
\begin{proof}
		It is straightforward to check that $(\rho,\sigma)$ is a twisting pair. Moreover 
		since the group is finite we don't have to worry about completions and the maximal reduced crossed product coincides with the reduced one. Let's check that $\Phi$ is a morphism of $C^*$-algebras. 		
		Then it will be invertible thanks to the rank-$1$ regular property with $\Phi^{-1}(b)$ given by writing $b=\sum_{k\in \widehat{G}}b_k \mu_k$.
		 We have
		\begin{displaymath}
			\begin{split}
				\Phi(a_j\delta_j \star b_k\delta_k)&=\Phi(a_j\rho_{j}(b_{k})\sigma_{j,k}\delta_{j+k})= a_{j}(\mu_{j}b_{k}\mu_{j}^{*})({\mu_j}\mu_k\,\mu_{j+k}^*)\mu_{j+k} =\Phi(a_j\mu_j)\Phi(b_k\mu_k)
			\end{split}
		\end{displaymath}
		and 
		\begin{displaymath}
			\Phi((a_j\delta_j)^*)=\Phi(\sigma_{-j, j}^{*}\rho_{-j}(a_{j}^{*})\delta_{-j})= \mu_{e}\mu_{j}^{*}\mu_{-j}^{*}(\mu_{-j}a_{j}^{*}\mu_{-j}^{*})\mu_{-j} = \mu_j^{*} a_j^*=\Phi(a_j\delta_j)^*.
		\end{displaymath} 
Concerning the statement about the equivariance, this follows from
$$\Phi(\gamma_g(a_j\delta_j))=\Phi(\langle j,g\rangle a_j\delta_j)=\langle j,g\rangle a_j\mu_j=\gamma_g(a_j\mu_j)=\gamma_g(\Phi(a_j\delta_j)).$$
Here we denote with $\gamma_g$ both the actions.
\end{proof}

\begin{rmk}
Following \cite{busby1970representations}, let $B(\widehat{G},A)$ be the group of all the maps $p:\widehat{G} \rightarrow U(A)$ satisfying $p(e)=1$. 
There is a natural action of $B(\widehat{G},A)$ on the set of  twisting pairs,  where $p \in B(\widehat{G},A)$ acts on $(\rho,\sigma)$  by $p \cdot (\rho,\sigma):=(\rho^p,\sigma^p)$, with 
$$
\rho^p_j=\operatorname{Ad}_{p(j)}\circ \rho_j, \quad \textrm{and} \quad \sigma_{j,k}^p=p(j)\rho_x(p(k))\sigma_{j,k}p(j+k)^*,
$$
for $j,k \in G$.
The isomorphism class of $A \rtimes^{\rho,\sigma} \, \widehat{G}$ only depends on the orbit of the twisting pair under the above $B(\widehat{G},A)$-action.
We see that any two frames give rise to twisting pairs {\em{in the same cohomology class}} i.e. in the same orbit. Indeed, if $\mu,\mu^{\prime}: \widehat{G} \rightarrow B$ are frames with corresponding twisting pairs $(\rho,\sigma)$ and $(\rho',\sigma')$, then putting $p(j):=\mu^{\prime*}_{j}\mu_j$ for $j\in \widehat{G}$, we find $(\rho',\sigma')=p \cdot (\rho,\sigma)$.
\end{rmk}

\section{A covering of the quantum $2$-torus}
	
Let $\theta$ be an irrational number so that the (skew adjoint) matrix
$
\Theta:=
\begin{pmatrix}
	0 &  \theta \\
	-\theta  & 0
\end{pmatrix}
$ 
defines the skew bicharacter
$$
\sigma_{\Theta}:(x,y)\in\mathbb{Z}^2 \times \mathbb{Z}^2 \longrightarrow \exp(-\pi i\langle x,\Theta y\rangle ) = e^{i\pi\theta(x_2y_1-x_1y_2)} \in U(1).
$$
This is a cocycle for a twisted action $(\operatorname{Id},\sigma)$ on $\mathbb{C}$ giving rise to the quantum $2$ torus
$$
\mathbb{A}^2_{\theta}=\mathbb{C}\rtimes^{\sigma}\mathbb{Z}^2.
$$
In particular this is the universal $C^*$ algebra generated by unitaries $U$ and $V$ satisfying the commutation relation $VU=e^{2\pi i \theta} UV$. We set $\zcW_x := e^{\pi i\theta x_1x_2} U^{x_1}V^{x_2}$, for any $x\in \mathbb{Z}^2$, so that, for any $y\in\bz^2$, we get
\begin{align*}
\zcW_x^* & = \zcW_{-x}, \\
\zcW_x \zcW_y & = \sigma_{\Theta}(x,y) \zcW_{x+y}.
\end{align*}

We denote by $B$ the C$^*$-algebra generated by $\{ \zcW_x : x\in\bz^2 \}$, and recall the construction of its unique trace. Let $\cs(\bz^2) := \{ \a\in\ell^\infty(\bz^2) : \sup_{r\in\bz^2} (1+r_1^2+r_2^2)^k |\a_r|^2 < \infty, \forall k\in\bn \}$, and $\ct_B := \{ \sum_{r\in\bz^2} \a_r \zcW_r : \{\a_r\}_{r\in\bz^2} \in\cs(\bz^2) \}$, which is a dense $*$-subalgebra of $B$ (called the smooth noncommutative torus), and set, for any $\sum_{r\in\bz^2} \a_r \zcW_r \in\ct_B$, $\t_B(\sum_{r\in\bz^2} \a_r \zcW_r) :=\a_{00}$. This extends to a tracial state on $B$. 

We denote with $L^2(B,\t_B)$ the Hilbert space of the GNS representation of $\t_B$, and with $\iota_B : B \hookrightarrow L^2(B,\t_B)$ the canonical immersion.

Then $\{ \iota_B(\zcW_x) : x\in\bz^2\}$ is an orthonormal basis of $L^2(B,\t_B)$, and, 
for any $\a\in\cs(\bz^2)$, we get $\| \sum_{r\in\bz^2} \a_r \iota_B(\zcW_r) \|_{L^2(B,\t_B)}^2 = \sum_{r\in\bz^2} |\a_r|^2$.

\np Let us introduce a finite abelian group of automorphisms of $B$. Let $M\in M_2(\bz)$, $\det M>1$, and set $\widehat{M} := (M^{-1})^T$. For any $t,x\in\bz^2$, set $\widetilde{\g}_t(\zcW_x) := e^{2\pi i \langle M^{-1}t, x \rangle} \zcW_x = e^{2\pi i \langle t, \widehat{M}x \rangle} \zcW_x$, which extends to an automorphism of $B$. If $t=Mt'$, $t'\in\bz^2$, then $\widetilde{\g}_{Mt'}(\zcW_x) = e^{2\pi i\langle t',x\rangle }\zcW_x = \zcW_x$, so that $\widetilde{\g}$ descends to an action of $G := \bz^2/M\bz^2$, given by $\g_g := \widetilde{\g}_t$, $\forall t\in\bz^2$ such that $t+M\bz^2 = g$. 

\np Since the dual group is $\widehat{G} \cong \widehat{M}\bz^2/\bz^2$, with duality given by $\langle \widehat{g}, g \rangle :=  \exp(2\pi i \langle t,\widehat{M} s \rangle)$, $\forall s,t\in\bz^2$ such that $t+M\bz^2 = g$, and $\widehat{M}s+\bz^2 = \widehat{g}$, we obtain the spectral subspaces of $\g$, which are 
$$
B_{\widehat{g}} := \{ b\in B : \g_{g}(b) = \langle \widehat{g},g \rangle b, \forall g\in G \} = \overline{\mathrm{span}}\{ \zcW_{s+M^Tn} : n\in\bz^2\}.
$$ 
In particular, $A:= B^\g$ is the C$^*$-algebra generated by $\{  \zcW_{M^Tn} : n\in\bz^2\}$.
		
\np We now choose, once and for all, a set $\widehat{\Gamma}$ in $\bz^2$, ``representing'' the elements of $\widehat{G}$. 
Let us consider the short exact sequence of groups
\begin{align}
& 0\longrightarrow \bz^2\longrightarrow \widehat{M}\bz^2 \longrightarrow \widehat{G} \longrightarrow 0. \label{seq0}
\end{align}
Such central extension $\widehat{M}\bz^2$ of $\widehat{G}$ via $\bz^2$ can be described either with a section $\zeta:\widehat{G}\to \widehat{M} \bz^2$ or via a $\bz^2$-valued 2-cocycle $\omega(\widehat{g},\widehat{h}) = \zeta(\widehat{g})+ \zeta(\widehat{h})- \zeta(\widehat{g}+\widehat{h})$. We choose the unique section such that, for any $\widehat{g} \in \widehat{G}$, $\zeta(\widehat{g})\in [0,1)^2$, and set $s(\widehat{g}) := M^T\zeta(\widehat{g}) \in \bz^2$, and $\widehat{\Gamma} := \{ s(\widehat{g}) :  \widehat{g} \in \widehat{G} \}$. 
Then $u(\widehat{g},\widehat{h}) := s(\widehat{g})+s(\widehat{h})-s(\widehat{g}+\widehat{h})\in M^T\bz^2$.
We set $\mu_{ \widehat{g} } := \zcW_{ s(\widehat{g}) }$, for all $\widehat{g} \in \widehat{G}$, so that $\mu_{ \widehat{g} } \in B_{ \widehat{g} }$. That is, $\mu$ is a frame for $G \curvearrowright B \supset A$. There results a twisting pair $(\rho,\sigma)$ 
and an isomorphism 
$A\rtimes^{\rho,\sigma}_r \widehat{G} \cong B$.

\subsection{A triple on the twisted crossed product}
Recall the canonical spectral triples on $A$ and $B$,

\begin{defn} \label{Cliff5.05}
We consider the closable derivations $\delta^B_j$ on $\iota_B(\ct_B) \subset L^2(B,\t_B)$, $j=1,2$ verifying 
$\delta^B_j \big(  \iota_B(\zcW_r) \big) = 2\pi i  r_j \, \iota_B(\zcW_r)$, and denote by $\delta^A_j$ the restrictions to $\iota_A(\ct_A) \subset L^2(A,\t_A)$.
%\begin{align*}
%\delta^B_j \big( \sum_{r\in \bz^2} \a_r\, \iota_B(\zcW_r) \big) & := 2\pi i \sum_{r\in \bz^2} r_j \a_r\,  \iota_B(\zcW_r), \\ 
%\delta^A_j \big( \sum_{r\in M^T\bz^2} \a_r\,  \iota_A(\zcW_r) \big) & := 2\pi i \sum_{r\in M^T\bz^2} r_j \a_r\,  \iota_A(\zcW_r).
%\end{align*}
Introduce the Hilbert spaces 
$$
\ch_B := L^2(B,\t_B) \oplus L^2(B,\t_B), \qquad \ch_A := L^2(A,\t_A) \oplus L^2(A,\t_A),
$$
the representations 
\begin{align*}
\pi_B : b\in B \mapsto 
\begin{pmatrix}
\pi_{\t_B}(b) & 0 \\
0 & \pi_{\t_B}(b)
\end{pmatrix} \in \cb(\ch_B),
\qquad 
\pi_A : a\in A \mapsto 
\begin{pmatrix}
\pi_{\t_A}(a) & 0 \\
0 & \pi_{\t_A}(a)
\end{pmatrix} \in \cb(\ch_A),
\end{align*}
and the closed selfadjoint operators $D_A$, resp. $D_B$ which, on $\iota_A(\ct_A) \oplus \iota_A(\ct_A) \subset \ch_A$, resp. $\iota_B(\ct_B) \oplus \iota_B(\ct_B) \subset \ch_B$, are given by
\begin{equation*}
D_A  = -i 
\begin{pmatrix}
0 & \delta^A_1 -i \delta^A_2 \\
\delta^A_1 +i \delta^A_2 & 0
\end{pmatrix},
\qquad
D_B  = -i 
\begin{pmatrix}
0 & \delta^B_1 -i \delta^B_2 \\
\delta^B_1 +i \delta^B_2 & 0
\end{pmatrix},
\end{equation*}
which are the Dirac operators on $A$ and $B$, respectively, and the gradings
$$
\G_B := \begin{pmatrix}
I_{L^2(B,\t_B)} & 0 \\
0 & -I_{L^2(B,\t_B)}
\end{pmatrix} \in \cb(\ch_B),
\qquad
\G_A := \begin{pmatrix}
I_{L^2(A,\t_A)} & 0 \\
0 & -I_{L^2(A,\t_A)}
\end{pmatrix} \in \cb(\ch_A).
$$ 
Then $(\ct_B,\pi_B,\ch_B,D_B,\G_B)$ and $(\ct_A,\pi_A,\ch_A,D_A,\G_A)$ are even spectral triples on $B$ and $A$, respectively.
\end{defn}

%%%%%%%%%%%%%%%%%%%%%%%%%%%%%%%%%%

We take a matrix-valued length function on $\widehat{G}$ as follows. Consider the map
$$
\eps : (x_1,x_2)\in\br^2 \mapsto
\begin{pmatrix}
0 & x_1-ix_2 \\
x_1+ix_2 & 0
\end{pmatrix} \in M_2(\bc),
$$
which satisfies $\eps(x_1,x_2)^2 = (x_1^2+x_2^2)I_2$. Then a Clifford length function is defined as $\vecLength : \widehat{g} \in \widehat{G} \mapsto \eps(\s(\widehat{g})) \in M_2(\br)$, and we get a multiplication operator $M_{\vecLength} : \delta_{\widehat{g}} \otimes v \in \ell^2(\widehat{G})\otimes\bc^2 \mapsto \delta_{\widehat{g}} \otimes \vecLength(\widehat{g}) v \in \ell^2(\widehat{G})\otimes\bc^2$. Therefore we have a spectral triple $(\bc(\widehat{G}), \ell^2(\widehat{G})\otimes\bc^2, M_{\vecLength})$ on $C^*_r(\widehat{G})$.
In the description of the associated spectral triple on the twisted crossed product is useful to introduce the functions:
$\nu: \widehat{G}\times \widehat{G}\times \mathbb{Z}^2 \longrightarrow \mathbb{R}$ defined by
$$ \quad \nu(\widehat{j},\widehat{g},x)=\big{\langle} s(-\widehat{j}-\widehat{g}),\Theta\big{(}M^Tx+s(\widehat{j})-s(-\widehat{g})\big{)} \big{\rangle} + \big{\langle} M^Tx, \Theta\big{(}s(\widehat{j})-s(-\widehat{g})\big{)}\big{\rangle} - \big{\langle} s(\widehat{j}),\Theta s(-\widehat{g})\big{\rangle},$$
and $\mathcal{Y}:\widehat{G}\times \widehat{G}\longrightarrow \mathbb{Z}^2$ with $\mathcal{Y}(\widehat{j},\widehat{g})=\widehat{M}\big{(}s(-\widehat{j}-\widehat{g})+s(\widehat{j})-s(-\widehat{g})\big{)}.$

\begin{prop} \label{Cliff5.08.010}
The odd spectral triple on $A \rtimes^{\r,\s}_r \widehat{G} \cong B$, coming, as in Theorem \ref{spectraltripleconstruction}, from the spectral triples $(\ct_A,\pi_A,\ch_A,D_A,\G_A)$ on $A$, and $(\bc(\widehat{G}), \ell^2(\widehat{G})\otimes\bc^2, M_{\vecLength})$ on $C^*_r(\widehat{G})$, is  $(C(\widehat{G},\ct_A),\Pi,\ch,\widetilde{D})$, where:
\begin{itemize}
\item $C(\widehat{G},\ct_A) \cong \ct_B$ is the smooth subalgebra,

\item $\ch := \big( L^2(A,\t_A) \otimes \ell^2(\widehat{G}) \otimes \bc^2 \big) \oplus \big( L^2(A,\t_A) \otimes \ell^2(\widehat{G}) \otimes \bc^2 \big)$ is the Hilbert space
with the representation given by
\begin{align*}
\Pi(\zcW_{M^Tx} & \otimes \delta_{\,\widehat{j}}) 
\begin{pmatrix}
\iota_A( \zcW_m) \otimes \delta_{\,\widehat{g}} \otimes v \\
\iota_A( \zcW_n) \otimes \delta_{\,\widehat{h}} \otimes w
\end{pmatrix} \\
& =
\begin{pmatrix}
e^{-\pi i \nu(\widehat{j},\widehat{g},x)} \s_\Theta \big( M^T(x+\mathcal{Y}(\widehat{j},\widehat{g})),m \big)\,\iota_A\big{(}  \zcW_{m+M^T(x+\mathcal{Y}(\widehat{j},\widehat{g}))}\big{)} \otimes \delta_{\, \widehat{j}+\widehat{g} } \otimes v \\
e^{-\pi i \nu(\widehat{j},\widehat{h},x)} \s_\Theta \big( M^T(x+\mathcal{Y}(\widehat{j},\widehat{h})),n \big)\,\iota_A\big{(}  \zcW_{n+M^T(x+\mathcal{Y}(\widehat{j},\widehat{h}))}\big{)} \otimes \delta_{\,\widehat{j}+\widehat{h}} \otimes w
\end{pmatrix},
\end{align*}
where $m,n\in M^T\bz^2$, 
$\widehat{j},\widehat{g},\widehat{h}\in \widehat{G}$ and $x\in \mathbb{Z}^2$. 
\item Finally the Dirac operator is uniquely defined by the formula:
%
%\item 
\begin{align*}
\widetilde{D}  %& 
\begin{pmatrix}
\iota_A( \zcW_m) \otimes \delta_{\,\widehat{g}} \otimes v \\
\iota_A( \zcW_n) \otimes \delta_{\,\widehat{h}} \otimes w
\end{pmatrix} 
%\\ & = 
%\begin{pmatrix}
%	I_{L^2(A,\t_A)} \otimes M_{\vecLength} &  (-i \delta^A_1 -\delta^A_2) \otimes I_{\ell^2({G})} \otimes I_V \\
%	(-i \delta^A_1 +\delta^A_2) \otimes I_{\ell^2({G})} \otimes I_V  & -I_{L^2(A,\t_A)} \otimes  M_{\vecLength}
%\end{pmatrix}
%\begin{pmatrix}
%\iota_A( \zcW_m) \otimes \delta_{\widehat{g}} \otimes v \\
%\iota_A( \zcW_n) \otimes \delta_{\widehat{h}} \otimes w
%\end{pmatrix}
%\\ & 
=  
\begin{pmatrix}
	\iota_A( \zcW_m) \otimes \delta_{\,\widehat{g}} \otimes \vecLength(\widehat{g}) v + \iota_A( (-in_1-n_2)  \zcW_n)  \otimes \delta_{\,\widehat{h}} \otimes w \\
	\iota_A((-im_1+m_2)  \zcW_m)  \otimes \delta_{\,\widehat{g}} \otimes v - \iota_A( \zcW_n) \otimes \delta_{\,\widehat{h}} \otimes \vecLength(\widehat{h}) w
\end{pmatrix}.
\end{align*}
\end{itemize}
\end{prop}
\begin{proof}
It follows from Theorem \ref{spectraltripleconstruction} $(1b)$, and computations with the commutation relations.
\end{proof}

%%%%%%%%%%%%%%%%%%%%%%%%%%%%%%%%%%

\subsection{A triple on the twisted crossed product. $\widehat{G}$-equivariant case}

%An example of an  $\widehat{G}$-equivariant spectral triple on $A$ is given below.

\begin{prop} \label{}
Let $H:= \ch_B$, $\pi := \pi_B|_A$, and $U(\widehat{g}) := \pi_B(\mu_{\widehat{g}}) = \pi_B(\zcW_{s(\widehat{g})}) \in U(\ch_B)$. Then $(\ct_A,\pi,H,D_B)$ is a $\widehat{G}$-equivariant spectral triple on $A$.
\end{prop}
\begin{proof}
Computations with the commutation relations.
\end{proof}
In this case the equivariant construction (starting with an even triple) produces an odd spectral triple described in the following. 
%%%%%%%%%%%%%%%%%%%%%%%%%%%%%%%%%%

\begin{prop} \label{}
Introducing the Hilbert space $\widetilde{H} :=  (L^2(B,\t_B) \otimes \ell^{2}(\widehat{G})\otimes V) \oplus (L^2(B,\t_B) \otimes \ell^{2}(\widehat{G})\otimes V)$, the representation  
$\hat{\pi}\times \hat{L}: A \rtimes_r^{\rho,\sigma} \widehat{G} \rightarrow \mathbb{B}(\widetilde{H})$ determined by
\begin{align*}
\hat{\pi}\times \hat{L}(\zcW_{M^Tx}\delta_{\,\widehat{j}}) & 
\begin{pmatrix}
\iota_B( \zcW_m) \otimes \delta_{\,\widehat{g}} \otimes v \\
\iota_B( \zcW_n) \otimes \delta_{\,\widehat{h}} \otimes w
\end{pmatrix} 
\\ & = 
\begin{pmatrix}
e^{-\pi i[ 
\langle M^T x,\Theta s(\widehat{j}) \rangle + 
\langle M^T x+s(\widehat{j}), \Theta m \rangle 
] } 
\iota_B\big{(}  \zcW_{m+M^Tx+s(\widehat{j})} \big{)} 
\otimes \delta_{\,\widehat{j}+\widehat{g}} \otimes v 
\\
 e^{-\pi i[ 
\langle M^T x,\Theta s(\widehat{j}) \rangle + 
\langle M^T x+s(\widehat{j}), \Theta n \rangle 
] } 
\iota_B \big{(}  \zcW_{n+M^Tx+s(\widehat{j})} \big{)} 
\otimes \delta_{\,\widehat{j}+\widehat{h}} \otimes w 
\end{pmatrix}, 
\end{align*}
for $n,m\in\bz^2$, and the self-adjoint operator
$$
\widetilde{D} =
\begin{pmatrix}
	I_{L^2(B,\t_B)} \otimes M_{\vecLength} &  (-i \delta^B_1 -\delta^B_2) \otimes I_{\ell^2({G})} \otimes I_V \\
	(-i \delta^B_1 +\delta^B_2) \otimes I_{\ell^2({G})} \otimes I_V  & -I_{L^2(B,\t_B)} \otimes  M_{\vecLength}
\end{pmatrix},
$$
so that, for $n,m\in\bz^2$,
\begin{align*}
\widetilde{D}  
\begin{pmatrix}
\iota_B( \zcW_m) \otimes \delta_{\widehat{g}} \otimes v \\
\iota_B( \zcW_n) \otimes \delta_{\widehat{h}} \otimes w
\end{pmatrix} 
 = 
\begin{pmatrix}
	\iota_B( \zcW_m) \otimes \delta_{\widehat{g}} \otimes \vecLength(\widehat{g}) v + 
	(-in_1-n_2)\iota_B(   \zcW_n)  \otimes \delta_{\widehat{h}} \otimes w \\
	(-im_1+m_2)\iota_B(   \zcW_m)  \otimes \delta_{\widehat{g}} \otimes v - \iota_B( \zcW_n) \otimes \delta_{\widehat{h}} \otimes \vecLength(\widehat{h})  w
\end{pmatrix},
\end{align*}
we get that 
$(C(\widehat{G},\ct_A),\hat{\pi}\times \hat{L},\widetilde{H},\widetilde{D})$ is a regular spectral triple on the twisted crossed product $A \rtimes_{r}^{\rho, \sigma} G$.
\end{prop}
\begin{proof}
Setting $C := \pi\times U(A \rtimes_r^{\r,\s} \widehat{G}) \equiv \pi_B(B)$, and $\mathcal{C}:=\{ \sum_{\widehat{g}\in\widehat{G}} a_{\widehat{g}} U(\widehat{g}) : a_{\widehat{g}} \in\ct_A \} \equiv \pi_B(\ct_B)$. The spectral triple $(\mathcal{C},\ch_B,D_B)$ on the $C^*$-algebra $C$ is regular, being the pullback with $\pi_B^{-1}$ of the spectral triple $(\ct_B,\pi_B,\ch_B,D_B)$, which is regular (see \cite{gracia2013elements}, pag. 543). Moreover, $(\bc(\widehat{G}),\ell^2(\widehat{G})\otimes V, M_{\vecLength})$ is a regular spectral triple on $C^*_r(\widehat{G})$, because of Proposition \ref{23}.
Therefore, by Theorem \ref{maintheorem}, 
$(C(\widehat{G},\ct_A),\hat{\pi}\times \hat{L},\widetilde{H},\widetilde{D})$ is a regular spectral triple on the twisted crossed product $A \rtimes_{r}^{\rho, \sigma} G$.
\end{proof}
\appendix

\section{Infinitesimal order of compact operators}\label{app1}

Let $T$ be a positive compact operator acting on a separable Hilbert space $\H$, denote by $\{\mu_n(T)\}_{n\geq0}$ the non-increasing sequence (with multiplicity) of its eigenvalues and set
\begin{displaymath}
	\lambda_t(T)\coloneqq \#\{n\geq0: \mu_n(T)> t\}
\end{displaymath} 
for $t>0$, cf. e.g. \cite{FaKo}.
%\begin{defn}\label{defn:inf-order}
%	We call \emph{infinitesimal order} of $T$  the number $o(T)=\inf\{s>0:\tr T^{s}<\infty\}$.
%\end{defn}
%Note that if $T$ is finite rank, $o(T)=0$. 
The following holds:
\begin{thm}\label{thm:inf-order}
	$$
	\inf\{s>0:\tr T^{s}<\infty\}=
	\left(
	\liminf_{n\to\infty}\frac{\log(\mu_{n}(T))}{\log(1/n)}
	\right)^{-1}=
	\limsup_{t\to0}\frac{\log(\lambda_{t}(T))}{\log (1/t)}
	=\limsup_{n\to\infty}\frac{\log(\lambda_{1/n}(T))}{\log n}.
	$$
\end{thm}
The first equality in the statement has been  proved in \cite{GuIs9} theorem 1.4. The second equality has been  proved in \cite{GuIs5} proposition 1.13 for general semifinite von~Neumann algebras, we give a Hilbert space proof here for the sake of simplicity. We first need a lemma.
\begin{lemma}\label{ausiliary}
	Let $f:(0,\infty)\to(0,\infty)$ be a right-continuous, non-increasing, piecewise-constant function. 
	\item{$(1)$} If $\lim_{t\to\infty}f(t)=0$ and the set of discontinuity points consists of an unbounded increasing sequence $x_n$, then
	$$
	\limsup_{t\to\infty}\frac{\log(1/f(t))}{\log t}
	=\limsup_n\frac{\log(1/f(x_n))}{\log x_{n}},
	\quad
	\liminf_{t\to\infty}\frac{\log(1/f(t))}{\log t}
	=\liminf_n\frac{\log(1/f(x_n))}{\log x_{n+1}}.
	$$
	\item{$(2)$} If $\lim_{t\to0}f(t)=+\infty$ and the set of discontinuity points consists of an infinitesimal decreasing sequence $x_n$, then
	$$
	\limsup_{t\to0}\frac{\log f(t)}{\log 1/t}
	=\limsup_n\frac{\log f(x_n)}{\log 1/x_{n-1}}
	,\quad
	\liminf_{t\to0}\frac{\log f(t)}{\log 1/t}
	=\liminf_n\frac{\log f(x_n)}{\log 1/x_{n}}.
	$$
\end{lemma}
\begin{proof}
	(1) Let $t_k$ be an increasing sequence such that $\exists \lim_k\frac{\log(1/f(t_k))}{\log t_{k}}$.
	Possibly passing to a subsequence, we may assume that, for any $n\in\N$, there is at most one $k$ such that $x_n\leq t_k<x_{n+1}$, denote by $n_k$ the indices for which  
	$x_{n_k}\leq t_k<x_{n_k+1}$. Since $f(t)$ is constant in $[x_{n_k},x_{n_k+1})$ and $1/\log t$ is decreasing, we have, for any $y_{n}\in[x_n,x_{n+1})$ such that $y_{n_k}\geq t_{k}$, the inequalities
	$$
	\frac{\log(1/f(y_{n_k}))}{\log y_{n_k}}
	\leq \frac{\log(1/f(t_k))}{\log t_{k}}
	\leq\frac{\log(1/f(x_{n_k}))}{\log x_{n_{k}}}.
	$$
	On the one hand we get
	$$
	\lim_k\frac{\log(1/f(t_k))}{\log t_{k}}
	\leq\limsup_{n}\frac{\log(1/f(x_{n_k}))}{\log x_{n_{k}}}
	\leq\limsup_n\frac{\log(1/f(x_n))}{\log x_{n}}
	\leq\limsup_{t\to\infty}\frac{\log(1/f(t))}{\log t}.
	$$
	On the other hand, 
	$$
	\liminf_{t\to\infty}\frac{\log(1/f(t))}{\log t}
	\leq\liminf_n\frac{\log(1/f(y_{n}))}{\log y_{n}}
	\leq\liminf_k\frac{\log(1/f(y_{n_k}))}{\log y_{n_k}}
	\leq \lim_k\frac{\log(1/f(t_k))}{\log t_{k}}.
	$$
	Finally, by choosing $y_n$ close enough to $x_{n+1}$ so that
	$\frac{\log x_{n+1}}{\log y_n}\to1$, we get
	$$
	\liminf_n\frac{\log(1/f(y_{n}))}{\log y_{n}}
	=\liminf_n \frac{\log(1/f(x_{n}))}{\log x_{n+1}}.
	$$
	Since $t_k$ may be chosen as to reach the $\liminf$ or the $\limsup$, part (1) follows.
	
	As for part (2),  let $t_k$ be a decreasing sequence such that $\exists \lim_k\frac{\log(f(t_k))}{\log 1/t_{k}}$. As above, we may assume that, for a suitable subsequence $x_{n_k}$, $x_{n_k}\leq t_k<x_{n_k-1}$, and consider a sequence $y_{n}\in[x_n,x_{n-1})$ such that $y_{n_k}> t_{k}$.
	Again, 
	$$
	\frac{\log f(x_{n_k})}{\log 1/x_{n_{k}}}
	\leq \frac{\log f(t_k)}{\log1/ t_{k}}
	\leq\frac{\log f(y_{n_k})}{\log 1/y_{n_k}}.
	$$
	The rest of the proof follows as in part (1).
\end{proof}

\begin{proof}[Proof of theorem \ref{thm:inf-order}]
	During this proof, we suppress the dependence on $T$.
	We first extend the sequence $n\geq0\to\mu_n$ to a right-continuous non increasing function $t\in[0,\infty)\to\mu(t)$ by posing $\mu(t)=\mu_n$ for $n\leq t< n+1$. We also write $\lambda(t)$ instead of $\lambda_t$.

%	Let us  prove the first equality. We observe first that
%	$\tr T^\alpha=\sum_{n\geq0}\mu_n(T)^\alpha=\int_0^\infty\mu(t)^{\alpha}dt$,
%	then set 
%	\begin{displaymath}
%		d=\left(\liminf_{n\to\infty}\frac{\log(\mu_{n}(T))}{\log(1/n)}\right)^{-1}=\left(\liminf_{t\to\infty}\frac{\log(\mu(t))}{\log(1/t)}\right)^{-1}
%	\end{displaymath} and $\Omega=\{\alpha>0:\int_0^\infty\mu(t)^\alpha dt<\infty\}$.  We prove that   $o(T)\leq d$ and that $d\leq o(T)$.
%	\begin{itemize}
%		\item To prove that $o(T)\leq d$, set $a(t)=\frac{\log1/\mu(t)}{\log(t)}$. In particular
%		$\mu(t)=t^{-a(t)}$ and $\liminf_{t\to\infty} a(t)=1/d$.  If 
%		$\alpha>d$, then $\liminf_{t\to\infty}\alpha a(t)=\alpha/d>1$, hence 
%		there exists $\beta>1$ such that $\alpha a(t)\geq\beta$ for $t$ 
%		sufficiently large.  Therefore
%		$$
%		\int_{0}^\infty\mu(t)^\alpha dt=\int_{0}^\infty t^{-\alpha a(t)}dt\leq const+
%		\int_{0}^\infty t^{-\beta}dt<\infty,
%		$$
%		which implies $\alpha\in\Omega$, namely $(d,\infty)\subset\Omega$.
%		\item On the other hand, to prove that $d\leq o(T)$ we may assume $d>0$, namely $1/d<\infty$.  
%		Now let $t_k\to\infty$ be such that
%		$\ell_k:=\frac 
%		{\log1/\mu(t_k)} {\log {t_k}} \to 1/d$.  We have $\mu(t_k) = t_k^{-\ell_k}$.  Let now 
%		$\alpha<d$, namely $\alpha \ell_k\to\alpha/d<1$, and 
%		choose $\eps>0$ such that $\alpha \ell_k\leq 1-\eps$ eventually.  Then
%		$$
%		\int_{0}^{t_k}\mu(t)^\alpha dt \geq t_k\mu(t_k)^\alpha = t_k\cdot t_k^{-\alpha 
%			\ell_k} \geq t_k^\eps\to\infty,\quad \text{ for } k\to\infty.
%		$$
%		Therefore $\alpha\leq o(T)$, i.e. $d\leq o(T)$.
%	\end{itemize}
	
The first equality in the statement is  proved in \cite{GuIs9} theorem 1.4.		We  prove the second equality.
Let us  note that $\mu(t)$ satisfies the hypotheses of part (1) of  lemma  \ref{ausiliary}, 
where the set of discontinuity points consists of a sequence $\{p_n,n\in\N\}\subseteq\N$,
while $\lambda(t)$ satisfies the hypotheses of part (2), where the set of discontinuity points consists of the sequence $\{\mu(p_n),n\in\N\}$. We also note that $\lambda\circ\mu(p_n)=p_n$.
	By lemma above,
	\begin{displaymath}
		\displaystyle
		\liminf_{t\to\infty}\frac{\log(1/\mu(t))}{\log t}
		=\liminf_n\frac{\log(1/\mu(p_n))}{\log p_{n+1}},
	\qquad 
		\displaystyle
		\limsup_{t\to0}\frac{\log(\lambda(t))}{\log 1/t}
		=\limsup_n\frac{\log (\lambda(\mu_{p_n}))}{\log 1/\mu_{p_n-1}}.
	\end{displaymath}
	Finally,
	$$
	\limsup_n\frac{\log (\lambda(\mu_{p_n}))}{\log 1/\mu_{p_n-1}}
	=\limsup_n\frac{\log p_n}{\log 1/\mu_{p_n-1}}
	=\left(\liminf_n\frac{\log 1/\mu_{p_n-1}}{\log p_n}\right)^{-1}.
	$$
	The third equality comes from lemma \ref{ausiliary}. 
\end{proof}

\section{Regularity of spectral triples}\label{regularity}
%\subsection{Regularity for Spectral Triples}\label{sec1} 
Let $(\mathcal{A}, H,D)$ be a spectral triple over a unital $C^{*}$-algebra $A$ and 
\begin{equation}\label{delta}
	\delta(T)   \coloneqq [\abs{D}, T]
\end{equation} be the unbounded derivation on the domain $\Dom(\delta)$ of elements $T\in \mathbb{B}(H)$  which preserve $\Dom(\abs{D})$ and for which $[\abs{D}, T]$ extends to a bounded operator on $H$. For $k\geq 2$ we define $\delta^{k}$ inductively on the domain $\Dom(\delta^{k})= \Set{T\in \Dom(\delta)\,|\, \delta(T)\in \Dom(\delta^{k-1})}$ and we set
\begin{displaymath}
	\Dom^{\infty}(\delta) = \bigcap_{k=1}^{\infty}\Dom(\delta^{k}).
\end{displaymath}

\begin{defn}We say that  $(\mathcal{A}, H,D)$ is \emph{regular} if both $\pi(\mathcal{A})$ and $[D,\pi(\mathcal{A})]$ belong to $\Dom^{\infty}(\delta)$. 
	\end{defn}

In general the operator $\abs{D}$ is not invertible, a property that would certainly make life easier when dealing with regularity . The following result shows that it is possible to replace $\abs{D}$ with the invertible bounded perturbation $\Delta^{\frac{1}{2}} \coloneqq (1+D^{2})^{\frac{1}{2}}$. 

\begin{thm}[cf. \cite{higson2003local,uuye2011pseudo}]A spectral triple $(\mathcal{A}, H,D)$ over a unital $C^{*}$-algebra $A$ is regular if and only if
	$\pi(\mathcal{A})$ and $[D,\pi(\mathcal{A})])$ belong to $\Dom^{\infty}([\Delta^{\frac{1}{2}}, \cdot\,])$.  
\end{thm}

The regularity of a spectral triple is closely related to the existence of a so-called algebra of generalized differential operators \cite{conneslocalindexformula,higson2006residue, higson2003local}.
We recall some basic facts 
 (mainly following the exposition of \cite{uuye2011pseudo}). In the next pages $\Delta$ will be a self-adjoint positive and invertible operator on a Hilbert space $H$ (then strictly positive). When dealing with a spectral triple we take $\Delta=1+D^2$, so that $\Delta$ is thought to be of order two.

\begin{defn}
	 The \emph{$\Delta$-Sobolev space} of order $s\in \mathbb{R}$, denoted $	W^{s}= W^{s}(\Delta) $, is the Hilbert space completion of $\Dom(\Delta^{\frac{s}{2}})$ with respect to the inner product
	\begin{equation}\label{scpro}
		\langle \xi, \eta \rangle_{W^{s}} \coloneqq \langle \Delta^{\frac{s}{2}}\xi, \Delta^{\frac{s}{2}}\eta\rangle_{H}
	\end{equation}
	for every $\xi,\eta\in H$. 
\end{defn}
For any $t\leq s$ there exists a constant $C$ such that
	\begin{displaymath}
		\norm{\xi}_{W^{t}}\leq C\norm{\xi}_{W^{s}}
	\end{displaymath}
	for  $\xi\in W^{t}(\Delta)$. In particular, there is a continuous inclusion 	$W^{s}\subseteq W^{t}$.	Moreover, $\Dom(\Delta^{\frac{s}{2}})$ is complete in the norm \eqref{scpro} for any $s\geq 0$. (see \cite{uuye2011pseudo}\label{18} for a proof).
\begin{defn}
	 The space of \emph{$\Delta$-smooth vectors} is
	\begin{displaymath}
		W^{\infty}\coloneqq \bigcap_{s\in \mathbb{R}}W^{s}= \bigcap_{n=0}^{\infty}W^{2n}=\bigcap_{n=0}^{\infty}\Dom(\Delta^{n}).
	\end{displaymath}
\end{defn}

As the notation suggests, the space $W^{\infty}$ is dense in $W^{s}$ for any $s\in \mathbb{R}$ (in particular, also in $W^{0}=H$).  We denote the space of linear maps $P\colon W^{\infty}\rightarrow W^{\infty}$ by $\End(W^{\infty})$. 

\begin{defn}
	We say that a linear map $P\colon W^{\infty}\rightarrow W^{\infty}$ has \emph{analytic order} at most $t\in \mathbb{R}$ if it extends by continuity to a bounded linear operator $P\colon W^{s+t}\rightarrow W^{s}$ for any $s\in \mathbb{R}$. The space of such operators is denoted with $\textup{Op}^{t}= \textup{Op}^{t}(\Delta)$. We further define
	\begin{displaymath}
	\textup{Op} =	\textup{Op}(\Delta)\coloneqq \bigcup_{t\in \mathbb{R}} \textup{Op}^{t}(\Delta).
	\end{displaymath}
\end{defn}

Notice in particular that operators with analytic order at most $0$ extend to bounded linear operators on $H=W^{0}$, allowing us to identify $\textup{Op}^{0}$ with a subspace of $\mathbb{B}(H)$. We then have the following fact.

\begin{lemma}[cf.\cite{uuye2011pseudo}]\label{lemma}Operators with finite analytic order form a filtered algebra:
	\begin{enumerate}
		\item 	$\textup{Op}^{s}\subseteq \textup{Op}^{t}$ for any $s\leq t$
		\item $\textup{Op}^{s}\cdot \textup{Op}^{t}\subseteq \textup{Op}^{s+t}$
	\end{enumerate}
	In particular, $\textup{Op}^{0}$ is a subalgebra of $\textup{Op}\subseteq \mathbb{B}(H)$. 
\end{lemma}

The algebra $\textup{Op}^{0}$ plays a central role in the regularity of a spectral triple.

\begin{prop}[cf. \cite{gracia2013elements}, lemma 10.22]\label{propreg} Let  $(\mathcal{A}, H,D)$ be a regular spectral triple over a unital $C^{*}$-algebra $A$ and $\Delta\coloneqq 1+D^{2}$. We have that both $\pi(\mathcal{A})$ and $[D,\pi(\mathcal{A})]$ are in $\textup{Op}^{0}(\Delta)$. 
	
\end{prop}

%\begin{proof}[Sketch of the proof]The idea is to show by induction 
	%that 
	%\begin{displaymath}
		%\Delta^{\frac{n}{2}}\pi(a)\Delta^{-\frac{n}{2}} = \sum_{j=0}^{n}(-1)^{j}\binom{n}{j}\delta^{j}(\pi(a))\mathfrak{b}(D)^{j}
	%\end{displaymath}
%	where $\delta$ is the derivation \eqref{delta} and $\mathfrak{b}(D)= D(1+D^{2})^{-\frac{1}{2}}$ is the bounded transform of $D$. In particular $\Delta^{\frac{n}{2}}\pi(a)\Delta^{-\frac{n}{2}}$ is bounded for any $n\in \mathbb{N}$ ($0$-included) and so 
	%\begin{displaymath}
		%\norm{\pi(a)\xi}_{W^{n}} = \norm{\Delta^{\frac{n}{2}}\pi(a)\xi} =  \norm{\Delta^{\frac{n}{2}}\pi(a)\Delta^{-\frac{n}{2}}\Delta^{\frac{n}{2}}\xi}\leq \norm{\Delta^{\frac{n}{2}}\pi(a)\Delta^{-\frac{n}{2}}}\norm{\xi}_{W^{n}}.
	%\end{displaymath}
%	By interpolation, one gets the thesis for general exponents. 
%\end{proof}

A theorem of Higson relates the notion of regularity with the existence of an algebra of generalised differential operators.

\begin{defn}[cf.\cite{higson2003local}, definition 5.1]	 An $\mathbb{N}$-filtered subalgebra $\mathcal{D}\subseteq \textup{Op}(\Delta)$ is called an algebra of \emph{generalized differential operators} (GDO) if it is closed under the derivation $[\Delta, \cdot]$ and satisfies
	\begin{displaymath}
		[\Delta,\mathcal{D}^{k}]\subseteq \mathcal{D}^{k+1}
	\end{displaymath}
	for any $k\in \mathbb{N}$. 
\end{defn}

\begin{thm}[cf. \cite{higson2003local}]\label{thmreg}A spectral triple $(\mathcal{A}, H,D)$ over a unital $C^{*}$-algebra $A$  is regular if and only if there exists an algebra of generalized differential operators (with respect to $\Delta=1+D^2)$ containing $\pi(\mathcal{A})$ and $[D,\pi(\mathcal{A})]$ in degree zero. 
\end{thm}

We maintain the notation $\Delta=1+D^2$; then observe that  $D\in \textup{Op}^{1}(\Delta)$. Indeed for any $\xi \in W^{\infty}$ and $s\in\mathbb{R}$ we have
	\begin{displaymath}
		\begin{split}
			\norm{D\xi}^{2}_{W^{s}} & = \langle D^{2}(1+D^{2})^{s}\xi,\xi\rangle = \langle \mathfrak{b}(D)^{2} (1+D^{2})^{s+1}\xi,\xi \rangle\\
			& = \norm{\mathfrak{b}(D)^{2}\Delta^{\frac{s+1}{2}}\xi}^{2}\leq \norm{\mathfrak{b}(D)}^{2}\norm{\xi}_{W^{s+1}}^{2}
		\end{split}
	\end{displaymath}
where $\mathfrak{b}(D)= D(1+D^{2})^{-\frac{1}{2}}$ is the bounded transform of $D$. Now define the $\mathbb{N}$-filtered algebra $\mathcal{E}\subseteq \End(W^{\infty})$ inductively by:
\begin{enumerate}
	\item $\mathcal{E}^{0}$ is the subalgebra generated by $\pi(\mathcal{A})$ and $[D,\pi(\mathcal{A})]$,
	\item $\mathcal{E}^{1}= \mathcal{E}^{0} + [\Delta, \mathcal{E}^{0}] + \mathcal{E}^{0}[\Delta, \mathcal{E}^{0}]$, 
	\item $\mathcal{E}^{k} = \mathcal{E}^{k-1} + \sum_{j=1}^{k-1}\mathcal{E}^{j}\mathcal{E}^{k-j} + [\Delta, \mathcal{E}^{k-1}] + \mathcal{E}^{0}[\Delta, \mathcal{E}^{k-1}]$, for $k\geq 2$.
\end{enumerate}
The following result shows that the $\mathbb{N}$-filtered algebra $\mathcal{E}$ is the minimal GDO algebra among those considered in theorem \ref{thmreg}.
\begin{thm}[cf. \cite{conneslocalindexformula,higson2003local,uuye2011pseudo}]\label{thm2}Let $(\mathcal{A}, H,D)$ be a spectral triple over $A$. 	\begin{enumerate}
		\item If  $(\mathcal{A}, H,D)$ is regular, then  $\mathcal{E}^{k}\subseteq \textup{Op}^{k}$ for any $k\geq 0$. 
		\item If $\pi(\mathcal{A})W^{\infty}\subseteq W^{\infty}$ and $\mathcal{E}^{k}\subseteq \textup{Op}^{k}$, for any $k\geq 0$, then $(\mathcal{A}, H,D)$ is regular. 
	\end{enumerate} 
\end{thm}

Thanks to Higson's theorem one can prove that the exterior product of regular triples is regular. 
Let us recall the construction of the exterior product of odd spectral triples.
This is the unique case we are concerned. For the general (graded) case the reader may consult \cite[p. 434]{connes1995noncommutative} and \cite{uuye3}.
Let $(\mathcal{A}_1,H_1,D_1)$ and $(\mathcal{A}_2,H_2,D_2)$ be two spectral triples over the $C^*$-algebras $A_1$ and $A_2$. 
The product spectral triple over $A_1\otimes A_2$ is the even spectral triple $\big{(}\mathcal{A}_1 \otimes_{\operatorname{alg}}\mathcal{A}_2,  (H_1 \otimes H_2) \oplus (H_1 \otimes H_2), D_1 \times D_2\big{)}$ constructed in the following way:
\begin{itemize}
\item the tensor product algebra is represented on the Hilbert space $(H_1 \otimes H_2) \oplus (H_1 \otimes H_2)$ with standard grading by two copies of the tensor product representation.
\item To construct the Dirac operator one first proves that the operator 
\begin{displaymath}
	\left(\begin{matrix}
		0 &  D_1\otimes 1_{H_2} -i1_{H_1} \otimes D_2\\
		D_1\otimes 	1_{H_2} +i1_{H_{1}}\otimes  D_2  & 0
	\end{matrix}\right) 
\end{displaymath} 
on its natural initial domain is essentially selfadjoint. Then, its closure, denoted with $D_1 \times D_2$ is the desired operator.
\end{itemize}
The following result is implicitly contained in \cite{conneslocalindexformula,higson2003local} and explicitly proven in \cite{uuye2011pseudo}. 
\begin{thm}
The product of regular spectral triples is regular.
%In the particular case of two odd, regular spectral triples $(\mathcal{A}_1,H_1,D_1)$ and $(\mathcal{A}_2,H_2,D_2)$,
 %the product triple 
 %is identified with the spectral triple
 %$\big{(}\mathcal{A}_1 \otimes_{\operatorname{alg}}\mathcal{A}_2,  (H_1 \otimes H_2) \oplus (H_1 \otimes H_2), D_1 \times D_2\big{)}$ having for Dirac operator the closure of  the essentially selfadjoint operator
%\begin{displaymath}
	%D_1 \times D_2=\left(\begin{matrix}
		%0 &  D_1\otimes 1_{H_2} -i1_{H_1} \otimes D_2\\
		%D_1\otimes 	1_{H_2} +i1_{H_{1}}\otimes  D_2  & 0
	%\end{matrix}\right).
%\end{displaymath} 
\end{thm}
\begin{proof}
The algebra of generalised differential operators sufficing for the regularity of the product is obtained by forming the product of any of the two algebras of the corresponding factors. The product of these operators has to take into account the gradings. 
Indeed to treat the even and odd case at the same footing, graded tensor products and multigradings (see \cite{higson2000analytic} and also \cite{uuye2011pseudo,uuye3}) are used.
For the odd $\times$ odd case the identification of the product triple defined in terms of graded tensor products with the triple
$\big{(}\mathcal{A}_1 \otimes_{\operatorname{alg}}\mathcal{A}_2,  (H_1 \otimes H_2) \oplus (H_1 \otimes H_2), D_1 \times D_2\big{)}$
described above is performed as follows: a regular odd spectral triple corresponds canonically to a regular $1$-multigraded triple. The product of two of them is a regular $2$-multigraded spectral triple which, by $2$-periodicity of multigraded triples exactly corresponds  to $\big{(}\mathcal{A}_1 \otimes_{\operatorname{alg}}\mathcal{A}_2,  (H_1 \otimes H_2) \oplus (H_1 \otimes H_2), D_1 \times D_2\big{)}$
(this involves a change of orientation in the Hilbert space $\mathbb{C} \oplus \mathbb{C})$.
 \end{proof}

	\begin{ack}The authors would like to thank S. Azzali and N. Higson  for some useful discussions.  This work was partially supported by 
		the ERC Advanced Grant 669240 QUEST \textquotedblleft Quantum Algebraic Structures and Models\textquotedblright\,.	D. G. and T. I. acknowledge the MIUR Excellence Department Project awarded to the Department of Mathematics, University of Rome Tor Vergata, CUP E83C18000100006 and the University of Rome Tor Vergata funding scheme ``Beyond Borders'', CUP E84I19002200005. 
\end{ack}

\bibliography{bibliography}{}
\bibliographystyle{plain}

\end{document}